\documentclass[a4paper, reqno, 10pt]{amsart}
\usepackage{amssymb}
\usepackage{extarrows}
\usepackage{bm}
\usepackage[centering]{geometry}
\usepackage{enumitem}
\usepackage{longtable}
\usepackage{multirow}
\usepackage{array}
\usepackage{graphics}

\usepackage[all]{xy}

\numberwithin{equation}{section}
\theoremstyle{definition}
\newtheorem{defn}{Definition}[section]

\theoremstyle{plain}
\newtheorem{cor}[defn]{Corollary}
\newtheorem{thm}[defn]{Theorem}
\newtheorem{lem}[defn]{Lemma}

\newtheorem{prop}[defn]{Proposition}

\newtheorem{defn-thm}[defn]{Definition-Theorem}

\newtheorem*{Freyd-Mitchell}{Freyd-Mitchell embedding Theorem}

\def\Hom{\operatorname{Hom}}

\def\Ho{\operatorname{Ho}}
\def\cal{\mathcal}
\def\X{\mathcal{X}}

\def\A{\mathcal{A}}
\def\E{\mathbb{E}}

\title[Homotopy categories of admissible model structures]{Homotopy categories \\ of admissible model structures \\ on extriangulated categories}
\author[Shun-Jie Li, Yang Gao, Pu Zhang] {Shun-Jie Li, Yang Gao, Pu Zhang$^*$ \\ \\  School of Mathematical Sciences \\
Shanghai Jiao Tong University,  \ Shanghai \ 200240, \ China }
\thanks{$^*$ Corresponding author}
\thanks{lishunjie$\symbol{64}$sjtu.edu.cn \ \ \ \ gyang112358$\symbol{64}$sjtu.edu.cn \ \ \ \ pzhang$\symbol{64}$sjtu.edu.cn}

\begin{document}

\begin{abstract} The extriangulated category is a simultaneous generalization of exact categories and triangulated categories. H. Nakaoka and Y. Palu have proved that
the homotopy category of an admissible model structure on a weakly idempotent complete extriangulated category is a triangulated category.
Using the classic construction of distinguished triangles given by A. Heller and D. Happel,
this paper provides an alternative proof of Nakaoka - Palu Theorem. In fact,
the class $\Delta$ of distinguished triangles in the present paper and the class $\widetilde{\Delta}$ of distinguished triangles in \cite{NP} have the relation  $\Delta = - \widetilde{\Delta}$,
and hence the two triangulated structures on the homotopy category are isomorphic.

\vskip5pt


\vskip5pt


\end{abstract}

\maketitle

\vspace{-20pt} \section {\bf Introduction}

Extriangulated categories, introduced by H. Nakaoka and Y. Palu \cite{NP}, provide a simultaneous generalization of exact categories and triangulated categories.
There are indeed many extriangulated categories which are neither exact nor triangulated (see e.g. \cite [2.18, 3.30]{NP}).
Recent progress shows that extriangulated categories have provided a unified framework and a powerful approach for studying problems
in homological algebra, Auslander-Reiten theory, tilting theory, Gorenstein homological algebra, and cotorsion pairs etc. (see e.g. \cite{CZZ, HZZ, ZZ1,  HLN1, NOS, INP}).

\vskip5pt

The model structures and their homotopy categories, introduced by D. Quillen \cite{Q1}, are widely used in algebra and topology,
especially for representation theory, homological algebra, $K$-theory and homotopy theory (see e.g. \cite{Q2, Q3, H1, MV, HSS, Hir,  G2}).
M. Hovey \cite{H2} have proposed abelian model structures and discovered a one-to-one correspondence between abelian model structures and Hovey triples.
A. Beligiannis and I. Reiten \cite{BR} have constructed weakly projective model structures and
discovered a one-to-one correspondence between weakly projective model structures and hereditary complete cotorsion pairs with contravariantly finite core (see also \cite{CLZ}).
Both the model structures are defined on abelian categories, and the two correspondences have established a connection between model structures and representation theory, via cotorsion pairs.

\vskip5pt

Hovey's correspondence has been extended to weakly idempotent complete exact categories by J. Gillespie \cite{G1},
and to triangulated categories with a proper class by X. Y. Yang \cite{Y}.
Nakaoka and Palu \cite[Section 5]{NP} further extends this correspondence to weakly idempotent complete extriangulated categories.
More surprisingly, they prove that the homotopy category of an admissible model structure on a weakly idempotent complete extriangulated category is triangulated. See \cite[Theorem 6.20]{NP}.
Since the two cotorsion pairs associated with an admissible model structure are not necessarily hereditary,
and the corresponding cofibrant - fibrant objects do not necessarily form a Frobenius category,
Nakaoka - Palu Theorem is remarkable, and it is previously unknown even for abelian model structures.

\vskip5pt

For a model structure $(\mathrm{CoFib}, \ \mathrm{Fib}, \ \mathrm{Weq})$ on a category $\A$,
Quillen's {\it homotopy category} $\mathrm{Ho}(\mathcal{A})$ is by definition the localization category $\mathcal{A}[\mathrm{Weq}^{-1}]$ of $\A$
with respect to the class $\mathrm{Weq}$ of weak equivalences. Denote by $\gamma: \mathcal{A}\longrightarrow \mathrm{Ho}(\mathcal{A})$ the localization functor.
Quillen \cite[I, 3.10, Remark]{Q1} defined functors $\Sigma, \ \Omega: \mathrm{Ho}(\mathcal{A})\longrightarrow  \mathrm{Ho}(\mathcal{A})$, and pointed out that they satisfy
the one-sided version of the axioms of a triangulated category.
However, in general $\Sigma$ and $\Omega$ are not auto-equivalences of $\mathrm{Ho}(\mathcal{A})$, and $\mathrm{Ho}(\mathcal{A})$ is not triangulated.
In fact, $\mathrm{Ho}(\mathcal{A})$ is pre-triangulated in the sense of \cite{BR}: i.e.,
$(\Sigma, \Omega)$ is an adjoint pair,  $(\mathrm{Ho}(\mathcal{A}), \Omega)$ is left triangulated, and $(\mathrm{Ho}(\mathcal{A}), \Sigma)$ is right triangulated,
together with some compatible conditions.

\vskip5pt

For an admissible model structure on a weakly idempotent complete extriangulated category $(\mathcal{A}, \ \mathbb{E}, \ \mathfrak{s})$, 
Nakaoka and Palu \cite{NP} have reconstructed the functors $\Sigma$ and $\Omega$, 
by using the completeness of cotorsion pairs $(\mathcal{C}, \ \mathcal{F}\cap\mathcal{W})$ and $(\mathcal{C}\cap\mathcal{W}, \ \mathcal{F})$, 
introduced a group homomorphism $l: \mathbb{E}(Z, X) \longrightarrow \mathrm{Ho}(\mathcal{A})(Z, \Sigma X)$,
and then proved that $\Sigma$ and $\Omega$ are quasi-inverse of each other.
They have defined a standard triangle
$$X\overset{\gamma(f)}{\longrightarrow}Y\overset{\gamma(g)}{\longrightarrow}Z\overset{l(\delta)}{\longrightarrow} \Sigma X$$
 in $\mathrm{Ho}(\mathcal{A})$ as the image under the localization functor $\gamma$ of an $\mathbb{E}$-triangle $X\overset{f}{\longrightarrow}Y\overset{g}{\longrightarrow}Z\overset{\delta}{\dashrightarrow}$ in $\mathcal A$,  together with the ``connecting morphism" $l(\delta): Z\longrightarrow \Sigma X$;
and finally they can prove that $(\mathrm{Ho}(\mathcal{A}), \ \Sigma, \ \widetilde{\Delta})$ is a triangulated category (\cite[Theorem 6.20]{NP}),
where $\widetilde{\Delta}$ is the class of triangles in $\mathrm{Ho}(\mathcal{A})$ isomorphic to the standard triangles.
When $\mathcal{A}$ is a weakly idempotent complete exact category, a more direct proof of Nakaoka - Palu Theorem can be found in Gillespie \cite[Theorem 6.34]{G2}.

\vskip5pt

A good theorem should allow different proofs. In this paper a standard triangle
in $\mathrm{Ho}(\mathcal{A})$ induced by a morphism $f:X\longrightarrow Y$  in $\mathcal A$  is defined to be
$$X\overset{\gamma(f)}{\longrightarrow}Y\overset{\gamma(g)}{\longrightarrow} C(f)\overset{\gamma(h)}{\longrightarrow}\Sigma X$$
which is induced by the commutative diagram in $\mathcal{A}$
\[\xymatrix@R=0.5cm{X \ar[r] \ar[d]_{f} & M \ar[r] \ar[d] & \Sigma X \ar@{-->}[r]^-{\delta_X} \ar@{=}[d] & {} \\
Y \ar[r]^{g} & C(f) \ar[r]^{h} & \Sigma X \ar@{-->}[r]^-{f_*\delta_X} & {}}\]
where $X\longrightarrow M\longrightarrow \Sigma X\overset{\delta_X}{\dashrightarrow}$ is the suspension sequence of $X$.
This construction of standard triangles is similar to the one in the stable category of a Frobenius category, by A. Heller \cite{Heller} and D. Happel  \cite{Hap1, Hap2}. The corresponding class of distinguished triangles is denoted by $\Delta$.
The main result Theorem \ref{thm: main} of this paper claims that
$(\mathrm{Ho}(\mathcal{A}), \ \Sigma, \ \Delta)$ is a triangulated category. This gives an alterative proof of  Nakaoka - Palu Theorem.

\vskip5pt

The class $\Delta$ of distinguished triangles in this paper and the class $\widetilde{\Delta}$ of distinguished triangles in \cite{NP} have the relation  $$\Delta = - \widetilde{\Delta}$$
i.e., a triangle $X\overset{f}{\longrightarrow}Y\overset{g}{\longrightarrow}Z\overset{h}{\longrightarrow}\Sigma X$ is in $\Delta$
if and only if $X\overset{-f}{\longrightarrow}Y\overset{-g}{\longrightarrow}Z\overset{-h}{\longrightarrow}\Sigma X$ is in $\widetilde{\Delta}$. Thus, one has a triangle-isomorphism
$(\mathrm{Ho}(\mathcal{A}), \ \Sigma, \ \Delta) \cong (\mathrm{Ho}(\mathcal{A}), \ \Sigma, \ \widetilde{\Delta}).$
See Theorem \ref{thm: tri-relation}.  It is well-known that $\Delta \ne  \widetilde{\Delta}$ in general.

\vskip5pt

Nakaoka - Palu Theorem  is a remarkable discovery, a key technique in the original proof in \cite[Theorem 6.20]{NP} is the introduction of the ``connecting morphism" which needs a long argument.
What is interesting is that the two kinds of distinguished triangles in $\mathrm{Ho}(\mathcal{A})$
are induced by $\mathbb{E}$-triangles in $\mathcal A$: a standard triangle
$X\overset{\gamma(f)}{\longrightarrow}Y\overset{\gamma(g)}{\longrightarrow}Z\overset{l(\delta)}{\longrightarrow} \Sigma X$ in \cite{NP}  is induced by the $\mathbb{E}$-triangle
$X\overset{f}{\longrightarrow}Y\overset{g}{\longrightarrow}Z\overset{\delta}{\dashrightarrow}$, while
a standard triangle $X\overset{\gamma(f)}{\longrightarrow}Y\overset{\gamma(g)}{\longrightarrow} C(f)\overset{\gamma(h)}{\longrightarrow}\Sigma X$ in this paper is induced by the $\mathbb{E}$-triangle
$Y\overset{g}{\longrightarrow} C(f)\overset{h}{\longrightarrow}\Sigma X\overset{f_*\delta_X}{\dashrightarrow}$.

\section{\bf Preliminaries}

\subsection{Extriangulated categories} Let $\mathcal{A}$ be an additive category with a biadditive functor $\mathbb{E}:\mathcal{A}^{\mathrm{op}}\times \mathcal{A} \longrightarrow \mathrm{Ab}$.
An element $\delta\in\mathbb{E}(Z,X)$ is called an $\mathbb{E}$-{\it extension}, and the zero element $0\in \mathbb{E}(Z,X)$ is called the {\it split $\mathbb{E}$-extension}.
For $\delta\in\mathbb{E}(Z,X)$, $f:X\longrightarrow X'$, $g:Z'\longrightarrow Z$,
one has $\mathbb{E}$-extensions $f_*\delta: =\mathbb{E}(Z,f)(\delta)\in\mathbb{E}(Z,X')$ and $g^*\delta: = \mathbb{E}(g,X)(\delta)\in\mathbb{E}(Z',X)$. Since $\mathbb{E}$ is a bifunctor, by definition
one has  $g^*f_*\delta=f_*g^*\delta \in \E(Z',X').$

\vskip5pt

Let $\delta_1\in \mathbb{E}(Z,X)$ and $\delta_2\in \mathbb{E}(Z',X')$. By the additivity of $\mathbb{E}$, one has a natural isomorphism
$$\mathbb{E}(Z\oplus Z', X\oplus X')\cong \mathbb{E}(Z,X)\oplus \mathbb{E}(Z,X')\oplus\mathbb{E}(Z',X)\oplus\mathbb{E}(Z',X').$$
Denote by $\delta_1\oplus\delta_2\in\mathbb{E}(Z\oplus Z', X\oplus X')$ the element corresponding to $(\delta_1,0,0,\delta_2)$ through this isomorphism. 

Two sequences  $X \overset{f}{\longrightarrow} Y \overset{g}{\longrightarrow} Z$ and $X \overset{f'}{\longrightarrow} Y' \overset{g'}{\longrightarrow} Z$ of morphisms in $\mathcal{A}$ are  {\it equivalent}, if there is an isomorphism $\varphi:Y \longrightarrow Y'$ such that $f' = \varphi f, \ g = g'\varphi$.
Denote by $[X \overset{f}{\longrightarrow} Y \overset{g}{\longrightarrow} Z]$ the equivalence class of $X \overset{f}{\longrightarrow} Y \overset{g}{\longrightarrow} Z$. Put
$$[X \overset{f}{\longrightarrow} Y \overset{g}{\longrightarrow} Z]\oplus[X' \overset{f'}{\longrightarrow} Y' \overset{g'}{\longrightarrow} Z'] = [X\oplus X'\overset{f\oplus f'}{\longrightarrow }Y\oplus Y'\overset{g\oplus g'}{\longrightarrow }Z\oplus Z'].$$

\begin{defn}[{\cite[2.9]{NP}}] \ Let $\mathcal{A}$ be an additive category with a biadditive functor $\mathbb{E}: \mathcal{A}^{\mathrm{op}}\times \mathcal{A} \longrightarrow \mathrm{Ab}$.
An {\it additive realization} $\mathfrak{s}$ of $\mathbb{E}$ is a correspondence, which maps any $\mathbb{E}$-extension $\delta\in \mathbb{E}(Z,X)$ to an equivalence class
$\mathfrak{s}(\delta) = [X\overset{f}{\longrightarrow}Y\overset{g}{\longrightarrow}Z]$, satisfying the following conditions:

\vskip5pt

$(1)$ \ Let $\delta\in \mathbb{E}(Z,X)$ and $\delta'\in \mathbb{E}(Z',X')$ be $\mathbb{E}$-extensions with $\mathfrak{s}(\delta)=[X\overset{f}{\longrightarrow}Y\overset{g}{\longrightarrow}Z]$ and $\mathfrak{s}(\delta')=[X'\overset{f'}{\longrightarrow}Y'\overset{g'}{\longrightarrow}Z']$. Then for any morphism $u:X \longrightarrow X'$ and $w:Z \longrightarrow Z'$ with $u_*\delta=w^*\delta'$, there exists a morphism $v:Y \longrightarrow Y'$ such that the following diagram commutes.
$$\xymatrix@R=0.6cm{X \ar[r]^{f} \ar[d]_{u} & Y \ar[r]^{g} \ar@{.>}[d]_{v} & Z \ar[d]^{w} \\X' \ar[r]^{f'} & Y' \ar[r]^{g'} & Z'}$$
In this case, the sequence $X\overset{f}{\longrightarrow}Y\overset{g}{\longrightarrow}Z$ of morphism is called a {\it realization} of $\delta\in \mathbb{E}(Z,X)$, and denoted by $X\overset{f}{\longrightarrow}Y\overset{g}{\longrightarrow}Z\overset{\delta}{\dashrightarrow}$.

$(2)$ \ For the split $\mathbb{E}$-extension $0\in \mathbb{E}(Z,X)$,   $\mathfrak{s}(0)=[X\overset{\left(\begin{smallmatrix}1\\0\end{smallmatrix}\right)}{\longrightarrow}X\oplus Z\overset{\left(\begin{smallmatrix}0 & 1\end{smallmatrix}\right) }{\longrightarrow}Z].$

\vskip5pt

$(3)$ \ For $\mathbb{E}$-extensions $\delta_1\in \mathbb{E}(Z,X)$, $\delta_2\in \mathbb{E}(Z',X')$, one has $\mathfrak{s}(\delta_1\oplus \delta_2)=\mathfrak{s}(\delta_1)\oplus \mathfrak{s}(\delta_2)$.
\end{defn}

\begin{defn} [{\cite[Definition 2.12]{NP}}] \ Let $\mathcal{A}$ be an additive category. A triple $(\mathcal{A}, \ \mathbb{E}, \ \mathfrak{s})$ is an {\it extriangulated category}, if it satisfies the following conditions $(\mathrm{ET} 1)$, $(\mathrm{ET} 2)$, $(\mathrm{ET} 3)$, $(\mathrm{ET} 3)^{\mathrm{op}}$, $(\mathrm{ET} 4)$, $(\mathrm{ET} 4)^{\mathrm{op}}$.

\vskip5pt

$(\mathrm{ET} 1)$ \ \ $\mathbb{E}:\mathcal{A}^{\mathrm{op}} \times \mathcal{A} \longrightarrow \mathrm{Ab}$ is a biadditive functor.

\vskip5pt

$(\mathrm{ET} 2)$ \ \ $\mathfrak{s}$ is an additive realization of $\mathbb{E}$.

\vskip5pt

$(\mathrm{ET} 3)$ \ \ For any realizations of $\mathbb{E}$-extensions $\delta\in\mathbb{E}(Z,X)$ and $\delta'\in\mathbb{E}(Z',X')$:
$$\xymatrix@R=0.4cm{X \ar[r]^{f} \ar[d]_{u} & Y \ar[r]^{g} \ar[d]_{v} & Z \ar@{-->}[r]^{\delta} \ar@{.>}[d]_{w} & \\
					X' \ar[r]^{f'} & Y' \ar[r]^{g'} & Z' \ar@{-->}[r]^{\delta'} &}$$
if $vf=f'u$, then there exists $w:Z \longrightarrow Z'$ such that $wg=g'v$ and $u_*\delta=w^*\delta'$.

\vskip5pt

$(\mathrm{ET} 3)^{\mathrm{op}}$ \ \ For any realizations of $\mathbb{E}$-extensions $\delta\in\mathbb{E}(Z,X)$ and $\delta'\in\mathbb{E}(Z',X')$:
$$\xymatrix@R=0.4cm{X \ar[r]^{f} \ar@{.>}[d]_{u} & Y \ar[r]^{g} \ar[d]_{v} & Z \ar@{-->}[r]^{\delta} \ar[d]_{w} & \\
					X' \ar[r]^{f'} & Y' \ar[r]^{g'} & Z' \ar@{-->}[r]^{\delta'} &}$$
if $wg=g'v$, then there exists $u:X \longrightarrow X'$ such that $f'u=vf$ and $u_*\delta=w^*\delta'$.

\vskip5pt

$(\mathrm{ET} 4)$ \ \ For any realizations of $\mathbb{E}$-extensions $\delta\in\mathbb{E}(Z',X)$ and  $\varepsilon\in\mathbb{E}(X',Y)$:
$$ \xymatrix@R=0.6cm{
X \ar[r]^{f} \ar@{=}[d] & Y \ar[r]^{f'} \ar[d]_{g} & Z' \ar@{-->}[r]^{\delta} \ar@{.>}[d]^{d} & \\
X \ar@{.>}[r]^{h} & Z \ar@{.>}[r]^{h'} \ar[d]_{g'} & Y' \ar@{.>}[d]^{e} \ar@{-->}[r]^{\zeta} & \\
& X' \ar@{=}[r] \ar@{-->}[d]_{\varepsilon} & X' \ar@{-->}[d]^{\eta} & \\
& & & }$$
there are realizations of $\mathbb{E}$-extensions $\zeta\in\mathbb{E}(Y',X)$ and $\eta\in\mathbb{E}(X',Z')$, such that the diagram commutes and $d^*\zeta=\delta$, $f'_*\varepsilon=\eta$, $f_*\zeta=e^*\varepsilon$.

\vskip5pt

$(\mathrm{ET} 4)^{\mathrm{op}}$ \ \  For any realizations of $\mathbb{E}$-extensions $\zeta\in\mathbb{E}(Y',X)$ and $\eta\in\mathbb{E}(X',Z')$:
$$ \xymatrix@R=0.6cm{
X \ar@{.>}[r]^{f} \ar@{=}[d] & Y \ar@{.>}[r]^{f'} \ar@{.>}[d]_{g} & Z' \ar@{-->}[r]^{\delta} \ar[d]^{d} & \\
X \ar[r]^{h} & Z \ar[r]^{h'} \ar@{.>}[d]_{g'} & Y' \ar[d]^{e} \ar@{-->}[r]^{\zeta} & \\
& X' \ar@{=}[r] \ar@{-->}[d]_{\varepsilon} & X' \ar@{-->}[d]^{\eta} & \\
& & & }$$
there are realizations of $\mathbb{E}$-extensions $\delta\in\mathbb{E}(Z',X)$ and $\varepsilon\in\mathbb{E}(X',Y)$, such that the diagram commutes and $d^*\zeta=\delta$, $f'_*\varepsilon=\eta$, $f_*\zeta=e^*\varepsilon$.

\vskip5pt

If this is the case, a realization $X\overset{f}{\longrightarrow}Y\overset{g}{\longrightarrow}Z\overset{\delta}{\dashrightarrow}$ of an $\mathbb{E}$-extension $\delta\in \mathbb{E}(Z,X)$ is called an {\it $\mathbb{E}$-triangle},  
$f$  an {\it $\mathbb{E}$-inflation}, and $g$ an {\it $\mathbb{E}$-deflation}.
\end{defn}

\begin{thm}  [{\rm \cite[Corollary 3.12]{NP}}] \label{lem:homological fundamental}  \ Let $(\mathcal{A},\mathbb{E}, \mathfrak{s})$ be an extriangulated category. For any $\mathbb{E}$-triangle $X\overset{f}{\longrightarrow}Y\overset{g}{\longrightarrow}Z\overset{\delta}{\dashrightarrow}$ and object $W\in \mathcal{A}$, there are exact sequences of abelian groups$:$
$$
\mathcal{A}(W,X) \overset{\mathcal{A}(W,f)}{\longrightarrow}\mathcal{A}(W,Y)\overset{\mathcal{A}(W,g)}{\longrightarrow}\mathcal{A}(W,Z)\overset{(\delta_{\#})_W}{\longrightarrow}\mathbb{E}(W,X)\overset{\mathbb{E}(W,f)}{\longrightarrow}\mathbb{E}(W,Y)
\overset{\mathbb{E}(W,g)}{\longrightarrow}\mathbb{E}(W,Z)
$$
and $$
\mathcal{A}(Z,W) \overset{\mathcal{A}(g,W)}{\longrightarrow}\mathcal{A}(Y,W)\overset{\mathcal{A}(f,W)}{\longrightarrow}\mathcal{A}(X,W)\overset{(\delta^{\#})_W}{\longrightarrow}\mathbb{E}(Z,W)\overset{\mathbb{E}(g,W)}
{\longrightarrow}\mathbb{E}(Y,W)\overset{\mathbb{E}(f,W)}{\longrightarrow}\mathbb E(X,W)
$$
where $(\delta_{\#})_W \varphi:=\varphi^*\delta$ and $(\delta^{\#})_W \psi:=\psi_*\delta$.
\end{thm}

\begin{cor}\label{cor: factorization}\ Let $X\overset{f}{\longrightarrow}Y\overset{g}{\longrightarrow}Z \overset{\delta}{\dashrightarrow}$ be an $\mathbb{E}$-triangle in an extriangulated category $\A$.

\vskip5pt

$(1)$ \ For any morphism $\varphi:W\longrightarrow Z$ in $\mathcal{A}$, \  $\varphi^*\delta=0$ if and only if $\varphi$ factors through $g$.

\vskip5pt

$(2)$ \ For any morphism $\psi:X\longrightarrow W$ in $\mathcal{A}$, \  $\psi_*\delta=0$ if and only if $\psi$ factors through $f$.
\end{cor}

\begin{cor}[\rm {\cite[3.6]{NP}}]\label{cor:iso 2 of 3}  \ Let $(\mathcal{A},\mathbb{E}, \mathfrak{s})$ be an extriangulated category. Suppose that there is a commutative diagram of $\mathbb{E}$-triangles
$$\xymatrix@R=0.4cm{X \ar[r]^{f} \ar[d]_{u} & Y \ar[r]^{g} \ar[d]_{v} & Z \ar@{-->}[r]^{\delta} \ar[d]_{w} & \\
X' \ar[r]^{f'} & Y' \ar[r]^{g'} & Z' \ar@{-->}[r]^{\delta'} &}$$
with $u_*\delta=w^*\delta'$. If any two of $u$, $v$, $w$ are isomorphisms, then so is the third.
\end{cor}

Let $(\mathcal{A},\mathbb{E}, \mathfrak{s})$ be an extriangulated category. For any $\mathbb{E}$-triangle $X \overset{f}{\longrightarrow} Y \overset{g}\longrightarrow Z \dashrightarrow $ in $\cal A$, denote $Z$ by $\mathrm{Cone}(f)$, and 
$X$ by $\mathrm{CoCone}(g)$. By Corollary \ref{cor:iso 2 of 3}, $\mathrm{Cone}(f)$ and $\mathrm{CoCone}(g)$ are unique up to isomorphism.

\subsection{Diagram lemmas} 

\begin{prop}[{\cite[Proposition 3.15]{NP}}]\label{prop: pullback 1} \ Let $(\mathcal{A},\mathbb{E},\mathfrak{s})$ be an extriangulated category,  $X_1\overset{f_1}{\longrightarrow}Y_1\overset{g_1}{\longrightarrow}Z\overset{\delta_1}{\dashrightarrow}$ and $X_2\overset{f_2}{\longrightarrow}Y_2\overset{g_2}{\longrightarrow}Z\overset{\delta_2}{\dashrightarrow}$ be $\mathbb{E}$-triangles. Then there is a commutative diagram of $\mathbb{E}$-triangles
$$ \xymatrix@R=0.6cm{& X_2 \ar@{=}[r] \ar@{.>}[d]^{d_2} & X_2 \ar[d]^{f_2} & \\
X_1 \ar@{.>}[r]^{d_1} \ar@{=}[d] & W \ar@{.>}[r]^{e_1} \ar@{.>}[d]^{e_2} & Y_2 \ar@{-->}[r]^{\varepsilon_1} \ar[d]^{g_2} & \\
X_1 \ar[r]^{f_1} & Y_1 \ar[r]^{g_1} \ar@{-->}[d]^{\varepsilon_2} & Z \ar@{-->}[r]^{\delta_1} \ar@{-->}[d]^{\delta_2} & \\& & & \\}$$
such that $(g_2)^*\delta_1=\varepsilon_1$, $(g_1)^*\delta_2=\varepsilon_2$, $(d_1)_*\delta_1+(d_2)_{*}\delta_2=0$.
\end{prop}

We also need the dual of Proposition \ref{prop: pullback 1}.

\begin{prop}\label{prop: pushout version 1} \
Let $(\mathcal{A},\mathbb{E},\mathfrak{s})$ be an extriangulated category, and $X\overset{f_1}{\longrightarrow}Y_1\overset{g_1}{\longrightarrow}Z_1\overset{\delta_1}{\dashrightarrow},\ X\overset{f_2}{\longrightarrow}Y_2\overset{g_2}{\longrightarrow}Z_2\overset{\delta_2}{\dashrightarrow}$ be $\mathbb{E}$-triangles. Then there is a commutative diagram of $\mathbb{E}$-triangles
$$\xymatrix@R=0.6cm{X \ar[r]^{f_1} \ar[d]_{f_2} & Y_1 \ar[r]^{g_1} \ar@{.>}[d]^{d_2} & Z_1 \ar@{-->}[r]^{\delta_1} \ar@{=}[d] & \\
Y_2 \ar@{.>}[r]^{d_1} \ar[d]_{g_2} & W \ar@{.>}[r]^{e_1} \ar@{.>}[d]^{e_2} & Z_1 \ar@{-->}[r]^{\varepsilon_1} &\\
Z_2 \ar@{=}[r] \ar@{-->}[d]_{\delta_2} & Z_2 \ar@{-->}[d]^{\varepsilon_2}&& \\
& & & } $$
such that $ (f_2)_*\delta_1=\varepsilon_1$, $(f_1)_*\delta_2=\varepsilon_2$, $(e_1)^*\delta_1+(e_2)^{*}\delta_2=0. $
\end{prop}

\begin{prop}[{\cite[Proposition 3.17]{NP}}]\label{prop: pullback right comp} \ Let $(\mathcal{A},\mathbb{E},\mathfrak{s})$ be an extriangulated category,
and $X_1\overset{f_1}{\longrightarrow}Y_1\overset{g_1}{\longrightarrow}Z\overset{\delta_1}{\dashrightarrow}$, \ $X_2\overset{d_2}{\longrightarrow}W\overset{e_2}{\longrightarrow}Y_1\overset{\varepsilon_2}{\dashrightarrow}$ and  $X_1\overset{d_1}{\longrightarrow}W\overset{e_1}{\longrightarrow}Y_2\overset{\varepsilon_1}{\dashrightarrow}$ be  $\mathbb{E}$-triangles with $f_1=e_2d_1$. Then there is a commutative diagram of $\mathbb{E}$-triangles
$$ \xymatrix@R=0.6cm{& X_2 \ar@{=}[r] \ar[d]^{d_2} & X_2 \ar@{.>}[d]^{f_2} & \\
X_1 \ar[r]^{d_1} \ar@{=}[d] & W \ar[r]^{e_1} \ar[d]^{e_2} & Y_2 \ar@{-->}[r]^{\varepsilon_1} \ar@{.>}[d]^{g_2} & \\
X_1 \ar[r]^{f_1} & Y_1 \ar[r]^{g_1} \ar@{-->}[d]^{\varepsilon_2} & Z \ar@{-->}[r]^{\delta_1} \ar@{-->}[d]^{\delta_2} & \\& & & \\}$$
such that $(g_1)^*\delta_2=\varepsilon_2$, $(g_2)^*\delta_1=\varepsilon_1$, $(d_1)_*\delta_1+(d_2)_{*}\delta_2=0$.
\end{prop}

\subsection{Homotopy cartesian squares}

Similar to the homotopy cartesian squares in triangulated categories (\cite[Definition 1.4.1]{N}), one has

\begin{defn}[{\cite[Definition 3.1]{HJ}}] \  A commutative square in extriangulated category $\mathcal{A}$
$$\xymatrix@R=0.4cm{X \ar[r]^f \ar[d]_a & Y \ar[d]^-b \\ W \ar[r]^g & Z}$$
is {\it homotopy cartesian}, if $X\overset{\left(\begin{smallmatrix} f \\ a \end{smallmatrix}\right) }{\longrightarrow}Y\oplus W\overset{(b, -g)}{\longrightarrow}Z\overset{\delta}{\dashrightarrow}$ is an $\mathbb{E}$-triangle.
\end{defn}

\begin{prop}[{{\cite[3.2]{HJ} or \cite[1.20]{LN}}}] \label{prop: pushout mid complete homotopy square} \ Let $(\mathcal{A},\mathbb{E},\mathfrak{s})$ be an extriangulated category. If the both rows in the  diagram below are $\mathbb{E}$-triangles with $u_*\delta=\delta'$
$$\xymatrix@R=0.4cm{X \ar[r]^{f} \ar[d]_{u} & Y \ar[r]^{g} \ar@{.>}[d]_{v} & Z \ar@{-->}[r]^{\delta} \ar@{=}[d] & \\
X' \ar[r]^{f'} & Y' \ar[r]^{g'} & Z \ar@{-->}[r]^{\delta'} &} $$
then there exists a morphism $v$ such that the diagram commutes, and $X\overset{\left(\begin{smallmatrix} f \\ -u \\\end{smallmatrix}\right)}{\longrightarrow} Y\oplus X'\overset{(v, f')}{\longrightarrow}Y'\overset{(g')^*\delta}{\dashrightarrow}$ is an $\mathbb{E}$-triangle.
\end{prop}

Dually, one has

\begin{prop}\label{prop: pullback mid complete homotopy square} \ Let $(\mathcal{A},\mathbb{E},\mathfrak{s})$ be an extriangulated category. If the both rows in the diagram below are $\mathbb{E}$-triangles with  $w^*\delta'=\delta$
$$\xymatrix@R=0.4cm{X \ar[r]^{f} \ar@{=}[d] & Y \ar[r]^{g} \ar@{.>}[d]_{v} & Z \ar@{-->}[r]^{\delta} \ar[d]_{w} & \\
X \ar[r]^{f'} & Y' \ar[r]^{g'} & Z' \ar@{-->}[r]^{\delta'} &} $$
then there exists a morphism $v$ such that the diagram commutes, and $Y\overset{\left(\begin{smallmatrix} g \\ v \\\end{smallmatrix}\right)}{\longrightarrow} Z\oplus Y'\overset{(-w, g')}{\longrightarrow}Z'\overset{f_*\delta'}{\dashrightarrow}$ is an $\mathbb{E}$-triangle.
\end{prop}

\begin{cor}[{\cite[3.16]{NP}}]\label{cor: inflation and deflation matrix} \ Let $(\mathcal{A},\mathbb{E},\mathfrak{s})$ be an extriangulated category. Then

\vskip5pt

 $(1)$ \ For an $\mathbb{E}$-inflation $f:X \longrightarrow Y$ and a morphism $a:X \longrightarrow X'$,  $\left(\begin{smallmatrix} f\\a \end{smallmatrix}\right)$ is an $\mathbb{E}$-inflation.

\vskip5pt

 $(2)$ \ For an $\mathbb{E}$-deflation $g:Y \longrightarrow Z$ and a morphism $b:Z' \longrightarrow Z$,  $(g, b)$ is an $\mathbb{E}$-deflation.
\end{cor}

As the dual of \cite[Theorem 3.3]{KLW} one has

\begin{prop}\label{prop: pullback right complete homotopy square} \ Let $(\mathcal{A},\mathbb{E},\mathfrak{s})$ be an extriangulated category. If the both rows in the diagram below are $\mathbb{E}$-triangles with $vf=f'$
$$\xymatrix@R=0.4cm{X \ar[r]^{f} \ar@{=}[d] & Y \ar[r]^{g} \ar[d]_{v} & Z \ar@{-->}[r]^{\delta} \ar@{.>}[d]_{w} & \\
X \ar[r]^{f'} & Y' \ar[r]^{g'} & Z' \ar@{-->}[r]^{\delta'} &} $$
then there exists a morphism $w$ such that the diagram commutes, and $w^*\delta'=\delta$, and moreover $Y\overset{\left(\begin{smallmatrix} g \\ v \\ \end{smallmatrix}\right)}{\longrightarrow} Z\oplus Y'\overset{(-w, g') }{\longrightarrow}Z'\overset{f_*\delta'}{\dashrightarrow}$ is an  $\mathbb{E}$-triangle.
\end{prop}

\begin{lem}\label{lem: extri iso} \ Let $(\mathcal{A},\mathbb{E},\mathfrak{s})$ be an extriangulated category.
Suppose that $X\overset{f}{\longrightarrow}Y\overset{g}{\longrightarrow}Z\overset{\delta}{\dashrightarrow}$ is an $\mathbb{E}$-triangle,
and that $\varphi:X\longrightarrow X'$ and $\psi:Z'\longrightarrow Z$ are isomorphisms in $\mathcal{A}$. Then \ $X'\overset{f\varphi^{-1}}{\longrightarrow}Y\overset{g}{\longrightarrow}Z\overset{\varphi_*\delta}{\dashrightarrow}$ and \ $X\overset{f}{\longrightarrow}Y\overset{\psi^{-1}g}{\longrightarrow}Z'\overset{\psi^*\delta}{\dashrightarrow}$ are  $\mathbb{E}$-triangles.
\end{lem}

\begin{proof} \ Assume that $\mathfrak{s}(\varphi_*\delta)=[X'\overset{x}{\longrightarrow}D\overset{y}{\longrightarrow}Z]$. By (ET2) there exists $\alpha:Y\longrightarrow D$ such that the following diagram of $\mathbb{E}$-triangles commutes:
$$\xymatrix@R=0.4cm{X \ar[r]^{f} \ar[d]_{\varphi} & Y \ar[r]^{g} \ar@{.>}[d]_{\alpha} & Z \ar@{-->}[r]^{\delta} \ar@{=}[d] & \\
X' \ar[r]^{x} & D \ar[r]^{y} & Z \ar@{-->}[r]^{\varphi_*\delta} &.}$$
By Corollary \ref{cor:iso 2 of 3},  $\alpha$ is an isomorphism. It follows from the commutative diagram
$$\xymatrix@R=0.4cm{X' \ar[r]^{f\varphi^{-1}} \ar@{=}[d] & Y \ar[r]^{g} \ar@{.>}[d]_{\alpha} & Z \ar@{=}[d] \\
X' \ar[r]^{x} & D \ar[r]^{y} & Z} $$ that
$\mathfrak{s}(\varphi_*\delta)=[X'\overset{f\varphi^{-1}}{\longrightarrow}Y\overset{g}{\longrightarrow}Z]$.  Similarly, 
$X\overset{f}{\longrightarrow}Y\overset{\psi^{-1}g}{\longrightarrow}Z'\overset{\psi^*\delta}{\dashrightarrow}$ is an  $\mathbb{E}$-triangle.
\end{proof}

\begin{prop}\label{prop: ET4 version homotopy}\  Let $(\mathcal{A},\mathbb{E},\mathfrak{s})$ be an extriangulated category,  and $X\overset{f}{\longrightarrow}Y\overset{f'}{\longrightarrow}Z'\overset{\delta}{\dashrightarrow}$, \ $X\overset{h}{\longrightarrow}Z\overset{h'}{\longrightarrow}Y'\overset{\zeta}{\dashrightarrow}$ and $Y\overset{g}{\longrightarrow}Z\overset{g'}{\longrightarrow}X'\overset{\varepsilon}{\dashrightarrow}$ be $\mathbb{E}$-triangles with $h=gf$. Then there is a commutative diagram of $\mathbb{E}$-triangles
$$ \xymatrix@R=0.6cm{
X \ar[r]^{f} \ar@{=}[d] & Y \ar[r]^{f'} \ar[d]_{g} & Z' \ar@{-->}[r]^{\delta} \ar@{.>}[d]^{d} & \\
X \ar[r]^{h} & Z \ar[r]^{h'} \ar[d]_{g'} & Y' \ar@{.>}[d]^{e} \ar@{-->}[r]^{\zeta} & \\
& X' \ar@{=}[r] \ar@{-->}[d]_{\varepsilon} & X' \ar@{-->}[d]^{\eta} & \\
& & & }$$
such that $d^*\zeta=\delta$, $f'_*\varepsilon=\eta$, $f_*\zeta=e^*\varepsilon$. Moreover $Y\overset{\left(\begin{smallmatrix} f' \\ g \\\end{smallmatrix}\right)}{\longrightarrow} Z'\oplus Z\overset{(-d,  h')}{\longrightarrow}Y'\overset{f_*\zeta}{\dashrightarrow}$ is an $\mathbb{E}$-triangle.
\end{prop}

\begin{proof}\ By Proposition \ref{prop: pullback right complete homotopy square} there exists a morphism $d:Z'\longrightarrow Y'$ such that $df'=h'g$, $d^*\zeta=\delta$, and $Y\overset{\bigl(\begin{smallmatrix}
	f' \\ g \\
\end{smallmatrix}\bigr)}{\longrightarrow} Z'\oplus Z\overset{(-d, h')}{\longrightarrow}Y'\overset{f_*\zeta}{\dashrightarrow}$ is an $\mathbb{E}$-triangle. By Proposition \ref{prop: pullback right comp} one has the following commutative diagram of $\mathbb{E}$-triangles
$$\xymatrix@R=0.25cm@C=0.4cm{&& Z' \ar@{=}[rr] \ar[dd]_-{\left(\begin{smallmatrix}1 \\ 0\end{smallmatrix}\right)} && Z' \ar@{.>}[dd]^-{\tilde{d}} && \\ \\
Y \ar[rr]^-{\left(\begin{smallmatrix}f' \\ g\end{smallmatrix}\right)} \ar@{=}[dd] && Z'\oplus Z \ar[rr]^-{(-d, h')} \ar[dd]_{(0, 1)} && Y' \ar@{-->}[rr]^{f_*\zeta} \ar@{.>}[dd]^{e} && \\ \\
Y \ar[rr]^-{g} && Z \ar[rr]^-{g'} \ar@{-->}[dd]_{0} && X' \ar@{-->}[rr]^{\varepsilon} \ar@{-->}[dd]^{\tilde{\eta}} \ar@{-->}[dd]&& \\ \\
&&  && &&}$$
such that $f_*\zeta=e^*\varepsilon$, $(g')^*\tilde{\eta}=0$, $\bigl(\begin{smallmatrix} 1\\0\\ \end{smallmatrix}\bigr)_*\tilde{\eta}+\bigl(\begin{smallmatrix} f'\\g\\ \end{smallmatrix}\bigr)_*\varepsilon=0$. Thus $$(1,0)_*\bigl(\begin{smallmatrix} 1\\0\\ \end{smallmatrix}\bigr)_*\tilde{\eta}+(1,0)_*\bigl(\begin{smallmatrix} f'\\g\\ \end{smallmatrix}\bigr)_*\varepsilon=\tilde{\eta}+f'_*\varepsilon=0.$$  Put $\eta=-\tilde{\eta}$. Then $f'_*\varepsilon=\eta$. By Lemma \ref{lem: extri iso}, $Z'\overset{d}{\longrightarrow}Y'\overset{e}{\longrightarrow}X'\overset{\eta}{\dashrightarrow}$ is an $\mathbb{E}$-triangle. Then $eh'=g'$. This completes the poof.
\end{proof}

The following is a stronger version of Proposition \ref{prop: pushout version 1}.

\begin{prop}\label{prop: pushout homotopy}\ Let $(\mathcal{A},\mathbb{E},\mathfrak{s})$ be an extriangulated category, and $X\overset{f_1}{\longrightarrow}Y_1\overset{g_1}{\longrightarrow}Z_1\overset{\delta_1}{\dashrightarrow}$, $X\overset{f_2}{\longrightarrow}Y_2\overset{g_2}{\longrightarrow}Z_2\overset{\delta_2}{\dashrightarrow}$ be $\mathbb{E}$-triangles. Then there is a commutative commutative diagram of $\mathbb{E}$-triangles
$$\xymatrix@R=0.6cm{X \ar[r]^{f_1} \ar[d]_{f_2} & Y_1 \ar[r]^{g_1} \ar@{.>}[d]^{d_2} & Z_1 \ar@{-->}[r]^{\delta_1} \ar@{=}[d] & \\
Y_2 \ar@{.>}[r]^{d_1} \ar[d]_{g_2} & W \ar@{.>}[r]^{e_1} \ar@{.>}[d]^{e_2} & Z_1 \ar@{-->}[r]^{\varepsilon_1} &\\
Z_2 \ar@{=}[r] \ar@{-->}[d]_{\delta_2} & Z_2 \ar@{-->}[d]^{\varepsilon_2}&& \\
& & & } $$
such that $(f_2)_*\delta_1=\varepsilon_1$, $(f_1)_*\delta_2=\varepsilon_2$, $(e_1)^*\delta_1+(e_2)^{*}\delta_2=0$. Moreover,  $X\overset{\left(\begin{smallmatrix} f_1 \\ f_2 \\ \end{smallmatrix}\right) }{\longrightarrow}Y_1\oplus Y_2\overset{(d_2, -d_1) }{\longrightarrow}{W}\overset{e_1^*\delta_1}{\dashrightarrow}$ is an $\mathbb{E}$-triangle.
\end{prop}

\begin{proof} \ By Corollary \ref{cor: inflation and deflation matrix}, $\left(\begin{smallmatrix} f_1 \\ f_2 \end{smallmatrix}\right)$ is an $\mathbb{E}$-inflation. Thus one has an $\mathbb{E}$-triangle $X\overset{\left(\begin{smallmatrix} f_1 \\ f_2 \\ \end{smallmatrix}\right) }{\longrightarrow}Y_1 \oplus Y_2\overset{(d_2, -d_1)}{\longrightarrow}W\overset{\kappa}{\dashrightarrow}$.
By Proposition \ref{prop: pullback right comp} there are commutative diagrams of $\mathbb{E}$-triangles

\begin{minipage}[t]{0.45\linewidth}
\centering
$$\xymatrix@R=0.25cm@C=0.4cm{&& Y_1 \ar@{=}[rr] \ar[dd]_-{\left( \begin{smallmatrix}1 \\ 0\end{smallmatrix} \right)} && Y_1 \ar@{.>}[dd]^{d_2} && \\ \\
X \ar[rr]^-{\left( \begin{smallmatrix}f_1 \\ f_2\end{smallmatrix} \right)} \ar@{=}[dd] && Y_1\oplus Y_2 \ar[rr]^-{(d_2, -d_1)} \ar[dd]_-{(0, 1)} && W \ar@{-->}[rr]^{\kappa} \ar@{.>}[dd]^{\tilde{e_2}} && \\ \\
X \ar[rr]^{f_2} && Y_2 \ar[rr]^{g_2} \ar@{-->}[dd]_{0} && Z_2 \ar@{-->}[rr]^{\delta_2} \ar@{-->}[dd]^{\tilde{\varepsilon_2}} && \\ \\
&& && && }$$
\end{minipage}
\hfill
\begin{minipage}[t]{0.45\linewidth}
\centering
$$\xymatrix@R=0.25cm@C=0.4cm{&& Y_2 \ar@{=}[rr] \ar[dd]_-{\left( \begin{smallmatrix}0 \\ 1\end{smallmatrix} \right)} &&Y_2 \ar@{.>}[dd]^{\tilde{d_1}} &&\\ \\
X \ar[rr]^-{\left( \begin{smallmatrix}f_1 \\ f_2\end{smallmatrix} \right)} \ar@{=}[dd] && Y_1\oplus Y_2 \ar[rr]^-{(d_2, -d_1)} \ar[dd]_-{(1, 0)} && W \ar@{-->}[rr]^{\kappa} \ar@{.>}[dd]^{e_1} &&\\ \\
X \ar[rr]^{f_1} && Y_1 \ar[rr]^{g_1} \ar@{-->}[dd]_{0} && Z_1 \ar@{-->}[rr]^{\delta_1} \ar@{-->}[dd]^{\tilde{\varepsilon_1}} &&\\ \\
&& && &&}$$
\end{minipage}

\vskip5pt

\noindent where $\kappa=\tilde{e_2}^*\delta_2=e_1^*\delta_1$, \ $\bigl(\begin{smallmatrix} f_1 \\ f_2 \\ \end{smallmatrix}\bigr)_*\delta_2+\bigl(\begin{smallmatrix} 1 \\ 0 \\  \end{smallmatrix}\bigr)_*\tilde{\varepsilon_2}=0$, \ $\bigl(\begin{smallmatrix} f_1 \\ f_2 \\  \end{smallmatrix}\bigr)_*\delta_1+\bigl(\begin{smallmatrix} 0 \\ 1 \\  \end{smallmatrix}\bigr)_*\tilde{\varepsilon_1}=0$.
Thus  $$-\tilde{e_2}d_1=g_2, \ \ \tilde{d_1}=-d_1, \ \ e_1d_2=g_1.$$
Applying $(1, 0)_*$  to both sides of $\bigl(\begin{smallmatrix} f_1 \\ f_2 \\  \end{smallmatrix}\bigr)_*\delta_2+\bigl(\begin{smallmatrix} 1 \\ 0 \\  \end{smallmatrix}\bigr)_*\tilde{\varepsilon_2}=0$ one gets $(f_1)_*\delta _2+\tilde{\varepsilon _2}=0$;
and applying $(0, 1)_*$ to both sides of $\bigl(\begin{smallmatrix} f _1 \\ f _2 \\  \end{smallmatrix}\bigr)_*\delta _1+\bigl(\begin{smallmatrix} 0 \\ 1 \\  \end{smallmatrix}\bigr)_*\tilde{\varepsilon _1}=0$ one gets $(f _2)_*\delta _1+\tilde{\varepsilon _1}=0$.
Let $e_2=-\tilde{e_2}$, $\varepsilon_1=-\tilde{\varepsilon_1}$, $\varepsilon_2=-\tilde{\varepsilon_2}$. Then $$e_2d_1=g_2,  \ (f_1)_*\delta _2=\varepsilon _1,  \ (f _2)_*\delta _1=\varepsilon _2,  \ e _1^*\delta _1+e _2^*\delta _2=0.$$
By Lemma \ref{lem: extri iso}, $Y_1\overset{d_2}{\longrightarrow}W\overset{e_2}{\longrightarrow}Z_2\overset{\varepsilon_1}{\dashrightarrow}$ and $Y_2\overset{d_1}{\longrightarrow}W\overset{e_1}{\longrightarrow}Z_1\overset{\varepsilon_2}{\dashrightarrow}$ are $\mathbb{E}$-triangles. This completes the proof.
\end{proof}

The composition of two homotopy cartesian squares is also homotopy cartesian.

\begin{prop}[{\cite{HXZ}}]\label{prop: composition of homotopy squares}\ Let $(\mathcal{A},\mathbb{E},\mathfrak{s})$ be an extriangulated category. Consider the following commutative diagram in $\mathcal{A}$
$$\xymatrix@R=0.4cm{X \ar[r]^{f} \ar[d]_{u} & Y \ar[r]^{g} \ar[d]_{v} &Z \ar[d]_{w} \\
X' \ar[r]^{f'} & Y' \ar[r]^{g'} & Z'.} $$
If the two squares are homotopy cartesian in $\mathcal{A}$, then the outside rectangle is  homotopy cartesian.
\end{prop}

\subsection{Model structures} \ A morphism $f:X\longrightarrow Y$ is a {\it retract} of a morphism $g:X'\longrightarrow Y'$, provided that there is a commutative diagram
\[\xymatrix@R=0.4cm{X \ar[r]^{\varphi_1} \ar[d]_{f} & X' \ar[r]^{\psi_1} \ar[d]_{g} & X \ar[d]^{f} \\ Y \ar[r]^{\varphi_2} & Y' \ar[r]^{\psi_2}& Y}\]
such that $\psi_1\varphi_1=\mathrm{Id}_X$, $\psi_2\varphi_2=\mathrm{Id}_Y$.

\begin{defn}[{\cite{Q2}}]\label{def:model structure}\ A {\it model structure} on category $\mathcal{M}$ is a triple $(\mathrm{CoFib},\mathrm{Fib},\mathrm{Weq})$ of classes of morphisms,
in which the morphisms are called {\it cofibrations, fibrations, weak equivalences}, respectively, satisfying the following axioms:

\vskip5pt

(CM1) (Two out of three axiom) \  Let $X\stackrel f \longrightarrow Y\stackrel g \longrightarrow Z$ be morphisms in $\mathcal{M}$. If two of the morphisms $f,\ g,\ gf$ are weak equivalences, then so is the third.

\vskip5pt

(CM2) (Retract axiom) \  If $f$ is a retract of $g$ and $g$ is a cofibration (fibration, weak equivalence, respectively), then so is $f$.

\vskip5pt

(CM3) (Lifting axiom) \  Given a commutative square
$$\xymatrix@R=0.4cm{A\ar[r]^-a \ar[d]_-i & X \ar[d]^-p \\
B\ar[r]^-b \ar@{.>}[ru]^-s & Y }$$
with  $i\in \mathrm{CoFib}$ and $p\in \mathrm{Fib}\cap\mathrm{Weq}$, if either $i\in \mathrm{CoFib}\cap\mathrm{Weq}$ or $p\in\mathrm{Fib}$, then there exists a morphism $s: B\longrightarrow X$ such that $a=si,\ b=ps$.

\vskip5pt

(CM4) (Factorization axiom) \  Any morphism $f:X\longrightarrow Y$ has two factorizations  \ $f=pi = qj$, where $i\in \mathrm{CoFib}\cap\mathrm{Weq}, \ p\in \mathrm{Fib}$,   $j\in \mathrm{CoFib}$ and $q\in\mathrm{Fib}\cap\mathrm{Weq}$.

\vskip5pt

Set $\mathrm{TCoFib}:=\mathrm{CoFib}\cap \mathrm{Weq}$  and $\mathrm{TFib}:=\mathrm{Fib}\cap \mathrm{Weq}$.  Morphisms in $\mathrm{TCoFib}$ and in $\mathrm{TFib}$ are respectively called {\it trivial cofibrations} and {\it trivial fibrations}.
\end{defn}

Let $(\mathrm{CoFib}, \mathrm{Fib}, \mathrm{Weq})$ be a model structure on a category $\mathcal{M}$ with zero object. An object $X$ is {\it cofibrant}, if $0 \longrightarrow X$ is a cofibration. Denote by $\mathcal{C}$ the class of cofibrant objects.
An object $Y$ is {\it fibrant}, if $Y \longrightarrow 0$ is a fibration. Denote by $\mathcal{F}$ the class of fibrant objects.
An object $W$ is a {\it trivial object} if $0 \longrightarrow W$ is a weak equivalence, or equivalently, $W \longrightarrow 0$ is a weak equivalence. Denote by $\mathcal{W}$ the class of trivial objects.

\subsection{Admissible model structures and Hovey triples}

\begin{defn}[{\cite[Definition 5.5]{NP}}] \ A model structure $(\mathrm{CoFib},\ \mathrm{Fib},\ \mathrm{Weq})$ on an extriangulated category $(\mathcal{A},\mathbb{E}, \mathfrak{s})$ is {\it admissible} if the following conditions are satisfied.

\vskip5pt

(1) \  $\mathrm{CoFib}=\{\text{$\mathbb{E}$-inflation} \ f \mid  \mathrm{Cone}(f)\in \mathcal{C}\}$.

\vskip5pt

(2) \  $\mathrm{Fib}=\{\text{$\mathbb{E}$-deflation} \ f \mid  \mathrm{CoCone}(f)\in \mathcal{F}\}$.

\vskip5pt

(3) \  $\mathrm{TCoFib}=\{\text{$\mathbb{E}$-inflation} \ f \mid \mathrm{Cone}(f)\in \mathcal{C}\cap \mathcal{W}\}$.	

\vskip5pt

(4) \  $\mathrm{TFib}=\{\text{$\mathbb{E}$-deflation} \ f \mid \mathrm{CoCone}(f)\in \mathcal{F}\cap \mathcal{W}\}$.
\end{defn}

For a class $\mathcal X$ of objects of an extriangulated category $(\mathcal{A},\mathbb{E}, \mathfrak{s})$,  put $\mathcal X^{\bot}: = \{Y\in \mathcal{A}\mid \mathbb{E}(\mathcal{X}, Y)=0\}$, where
$\mathbb{E}(\mathcal{X}, Y)=0$ means that $\mathbb{E}(X, Y)=0$ for all $X\in \mathcal{X}.$ Similarly for $\mathbb{E}(Y, \mathcal{X})=0$ and  ${}^{\bot}\mathcal{X}$.

\begin{defn}[{ \cite[Definition 4.1]{NP}}] \ Let $(\mathcal{A},\mathbb{E}, \mathfrak{s})$ be an extriangulated category. A pair $(\mathcal{X}, \mathcal{Y})$  of classes of objects is a {\it cotorsion pair}, if
$\mathcal{X}={}^{\bot}\mathcal{Y}$ and $\mathcal{Y}=\mathcal{X}^{\bot}$.

\vskip5pt

A cotorsion pair $(\mathcal{X},\mathcal{Y})$ is {\it complete} if for any object $A\in \mathcal{A}$, there are $\mathbb{E}$-triangles
$$
A \longrightarrow Y \longrightarrow X \dashrightarrow,\qquad Y' \longrightarrow  X' \longrightarrow  A\dashrightarrow,
$$
where $Y \in \mathcal{Y},\ X \in \mathcal{X},\ X'\in \mathcal{X}, \ Y'\in \mathcal{Y}$.
\end{defn}

\begin{defn}\ Let $(\mathcal{A},\mathbb{E}, \mathfrak{s})$ be an extriangulated category. A triple $(\mathcal{C}, \mathcal{F}, \mathcal{W})$ of classes of objects of  $\mathcal{A}$ is a {\it Hovey triple}, if   $(\mathcal{C}\cap \mathcal{W}, \mathcal{F})$ and  $(\mathcal{C}, \mathcal{F}\cap \mathcal{W})$ are complete cotorsion pairs in $\mathcal{A}$, and
$\mathcal{W}$ satisfies the ``two out of three'' property for $\mathbb{E}$-triangles, i.e., whenever two out of three terms in an $\mathbb{E}$-triangle are in $\mathcal{W}$, so is the third.
\end{defn}

An additive category $\mathcal{A}$ is {\it weakly idempotent complete}, if any splitting monomorphism in $\mathcal{A}$ has a cokernel, or equivalently, if any splitting epimorphism in $\mathcal{A}$ has a kernel.

\begin{lem}[{\cite[Condition 5.8]{NP}, \cite[Proposition C]{K}}]\ Let $(\mathcal{A},\mathbb{E}, \mathfrak{s})$ be an extriangulated category. Then the following are equivalent.

\vskip5pt

$(1)$ \ Any splitting monomorphism in $\mathcal{A}$ has a cokernel.

\vskip5pt

$(2)$ \ Any splitting epimorphism in $\mathcal{A}$ has a kernel.

\vskip5pt

$(3)$ \ Any splitting monomorphism in $\mathcal{A}$ is an $\mathbb{E}$-inflation.

\vskip5pt

$(4)$ \ Any splitting epimorphism in $\mathcal{A}$ is an $\mathbb{E}$-deflation.

\vskip5pt

$(5)$ \ If $ki$ is an $\mathbb{E}$-inflation, then so is $i$.

\vskip5pt

$(6)$ \ If $de$ is an $\mathbb{E}$-deflation, then so is $d$.\end{lem}	

\vskip5pt

An extriangulated category satisfying the above equivalent conditions is called a {\it weakly idempotent complete} extriangulated category. Nakaoka and Palu \cite{NP} have extended Hovey's correspondence to extriangulated categories.

\begin{thm}[{\cite[Section 5]{NP}}]\ Assume $(\mathcal{A},\mathbb{E}, \mathfrak{s})$ is a weakly idempotent complete extriangulated category.
Then there is a one-to-one correspondence between admissible model structures on $\mathcal{A}$ and Hovey triples in $\mathcal{A}$, given by
$$(\mathrm{CoFib},\  \mathrm{Fib},\  \mathrm{Weq})\mapsto (\mathcal C, \mathcal F, \mathcal W)$$
where $\mathcal C$, $\mathcal F$, and $\mathcal W$ are the classes of cofibrant objects, fibrant objects, and trivial objects, respectively; and the inverse is given by
$$(\mathcal{C},\ \mathcal{F},\ \mathcal{W})\mapsto (\mathrm{CoFib},\  \mathrm{Fib},\  \mathrm{Weq})
$$
where $\mathrm{CoFib}=\{\text{$\mathbb{E}$-inflation}\ f\mid \mathrm{Cone}(f)\in \mathcal{C}\}$, \ $\mathrm{Fib} =\{\text{$\mathbb{E}$-deflation\ }f\mid \mathrm{CoCone}(f)\in \mathcal{F}\}$ and
$$\mathrm{Weq} =\{pi\mid i \ \text{is an $\mathbb{E}$-inflation}, \mathrm{Cone}(i)\in \mathcal{C}\cap \mathcal{W}, \ p\ \text{is an $\mathbb{E}$-deflation},  \mathrm{CoCone}(p)\in \mathcal{F}\cap \mathcal{W}\}.$$
\end{thm}

\subsection{The homotopy category of admissible model structures} \ Let $\mathcal{S}$ be a class of morphisms in a category $\mathcal{M}$.
By P. Gabriel - M. Zisman \cite{GZ}, the localization category $\mathcal{M}[\mathcal{S}^{-1}]$ of $\mathcal{M}$ with respect to $\mathcal{S}$ always exists.

\begin{lem}[{\cite[Lemma 2.2.1]{HK}}]\label{addlocal} \ Let $\mathcal{C}$ be an additive category and $\mathcal{S}$ a class of morphisms of $\mathcal{C}$.
Suppose that ${\rm Id}_X\in \mathcal{S}, \ \forall \ X\in\mathcal C$,  and that $\mathcal S$ is closed under coproducts, i.e., if $\sigma$ and $\tau$ are in $\mathcal{S}$ then $\sigma\oplus \tau$ is in $\mathcal{S}$.
Then $\mathcal{C}[\mathcal{S}^{-1}]$ is an additive category and the localization functor $\mathcal{C}\longrightarrow \mathcal{C}[\mathcal{S}^{-1}]$ is additive.
\end{lem}

\begin{defn}[\cite{Q1}] \ The {\it homotopy category} of a model structure $(\mathrm{CoFib},\  \mathrm{Fib},\  \mathrm{Weq})$ on a category $\mathcal{M}$ with zero object is the localization category $\mathcal{M}[\mathrm{Weq}^{-1}]$, which is denoted by $\mathrm{Ho}(\mathcal{M})$. Denote by $\gamma: \mathcal{M}\longrightarrow \mathrm{Ho}(\mathcal{M})$ the localization functor.
\end{defn}

Let $(\mathrm{CoFib},\  \mathrm{Fib},\  \mathrm{Weq})$ be a model structure on a weakly idempotent complete extriangulated category $(\mathcal{A},\mathbb{E}, \mathfrak{s})$. By Lemma \ref{addlocal}, $\mathrm{Ho}(\mathcal{A})$ is an additive category, and  $\gamma:\mathcal{A}\longrightarrow \mathrm{Ho}(\mathcal{A})$ is an additive functor.

\vskip5pt

\begin{thm}[{\cite[Theorem 1.1]{LZ}}]\label{QuillenThm} \ Let $(\mathrm{CoFib},\ \mathrm{Fib},\ \mathrm{Weq})$ be a model structure on a weakly idempotent complete additive category $\cal A$. Then the composition of the embedding $\mathcal{C}\cap\mathcal{F}\hookrightarrow \mathcal{A}$ and the localization functor $\gamma: \mathcal{A}\longrightarrow \mathrm{Ho}(\mathcal{A})$ induces an equivalence of categories $$(\mathcal{C}\cap\mathcal{F})/(\mathcal{C}\cap\mathcal{F}\cap\mathcal{W})\cong\mathrm{Ho}(\mathcal{A}),$$
where $\cal C$, $\cal F$, $\cal W$ are the classes of cofibrant objects, fibrant objects, and trivial objects, respectively.
\end{thm}

\begin{cor}\label{CcapF} \ Let $(\mathrm{CoFib},\ \mathrm{Fib},\ \mathrm{Weq})$ be an admissible model structure on a weakly idempotent complete extriangulated category $(\mathcal{A},\mathbb{E}, \mathfrak{s})$, and $(\cal C, \cal F, \cal W)$ be the corresponding Hovey triple. For any morphism $f:X\longrightarrow Y$ in $\mathrm{Ho}(\mathcal{A})$ with $X,\ Y\in \mathcal{C}\cap\mathcal{F}$, there exists a morphism $f':X\longrightarrow Y$ in $\mathcal{A}$ such that $f=\gamma(f')$.
\end{cor}

\begin{proof} \ By Theorem \ref{QuillenThm}, one has an equivalence of categories $(\mathcal{C}\cap\mathcal{F})/(\mathcal{C}\cap\mathcal{F}\cap\mathcal{W}) \longrightarrow \mathrm{Ho}(\mathcal{A})$, $[f]\mapsto \gamma(f)$, which preserves objects.
Then the assertion follows from the isomorphism $\mathrm{Hom}_{(\mathcal{C}\cap\mathcal{F})/(\mathcal{C}\cap\mathcal{F}\cap\mathcal{W})}(X,Y)\cong \mathrm{Hom}_{\mathrm{Ho}(\mathcal{A})}(X,Y)$, for $X,\ Y\in \mathcal{C}\cap\mathcal{F}$.
\end{proof}

Let $(\mathcal{C},\mathcal{F},\mathcal{W})$ be a Hovey triple in a weakly idempotent complete extriangulated category $(\mathcal{A},\mathbb{E}, \mathfrak{s})$.
Using the completeness of the cotorsion pairs $(\mathcal{C},\mathcal{F}\cap\mathcal{W})$ and  $(\mathcal{C}\cap\mathcal{W},\mathcal{F})$,  for any object $X\in \mathcal{A}$, one has the $\mathbb{E}$-triangles:
$$R_X \longrightarrow QX \overset{\rho_X}{\longrightarrow}X\dashrightarrow, \ \ \ \ \ \ QX \overset{\sigma_X}{\longrightarrow} RQX \longrightarrow Q_X\dashrightarrow$$ 

\noindent where $QX\in \mathcal{C}$, $R_X\in \mathcal{F}\cap \mathcal{W},\  RQX\in \mathcal{F},\  Q_X\in \mathcal{C}\cap \mathcal{W}$. Since $\mathcal{C}$ is closed under extensions, $RQX\in \mathcal{C}\cap \mathcal{F}$.
Note that $\rho_X\in \mathrm{TFib}$, $\sigma_X\in\mathrm{TCoFib}$. Thus one has an isomorphism in $\mathrm{Ho}(\mathcal{A})$:
$$\gamma(\sigma_X)\circ\gamma(\rho_X)^{-1}:X\longrightarrow RQX.$$

\noindent For each object $X\in \mathcal{A}$, one can choose two $\mathbb{E}$-triangles $R_X \longrightarrow QX \overset{\rho_X}{\longrightarrow}X\dashrightarrow$ and $QX \overset{\sigma_X}{\longrightarrow} RQX \longrightarrow Q_X\dashrightarrow$, and hence obtain a canonical isomorphism in $\mathrm{Ho}(\mathcal{A})$
$$\theta_X:=\gamma(\sigma_X)\circ\gamma(\rho_X)^{-1}:X\longrightarrow RQX.$$

The following observation is important in the proof of the main theorem.

\begin{prop}[{\cite[Section 5.3]{NP}}]\label{prop:canonical} \ For any morphism $f:X\longrightarrow Y$ in $\mathrm{Ho}(\mathcal{A})$, there exists a morphism $\tilde{f}:RQX\longrightarrow RQY$ in $\mathcal{A}$ such that the following square commutes in $\mathrm{Ho}(\mathcal{A})$
$$\xymatrix@R=0.4cm{X\ar[r]^f \ar[d]_-{\theta_X} & Y \ar[d]_-{\theta_Y} \\ RQX \ar[r]^{\gamma(\tilde{f})} & RQY,}$$
where $\theta_X$ are $\theta_Y$ are the canonical isomorphisms.
\end{prop}

In fact, consider the morphism $\theta_Y\circ f\circ \theta_X^{-1}:RQX\longrightarrow RQY$ in $\mathrm{Ho}(\mathcal{A})$. By Corollary \ref{CcapF}, there exists a morphism $\tilde{f}:RQX\longrightarrow RQY$ in $\mathcal{A}$ such that $\theta_Y\circ f\circ \theta_X^{-1}=\gamma(\tilde{f}).$

\section{\bf The homotopy category of an admissible model structure}

We will give an alternative proof of Nakaoka - Palu Theorem, which claims that the homotopy category of an admissible model structure on a weakly idempotent complete extriangulated category is triangulated.

\subsection{Two facts on admissible model structures} \ Let $(\mathrm{CoFib},\ \mathrm{Fib},\ \mathrm{Weq})$ be an admissible model structure on a weakly idempotent complete extriangulated category $(\mathcal{A},\mathbb{E}, \mathfrak{s})$, and $(\mathcal{C},\mathcal{F},\mathcal{W})$ the corresponding Hovey triple.

\vskip5pt

Using Proposition \ref{prop: pullback 1}, Proposition \ref{prop: pushout version 1}, and the definitions of admissible model structures and Hovey triples, one obtains the following lemma.

\begin{lem}[{\cite[5.12]{NP}}]\label{lem: weq} \ $(1)
	$ \  Let $f$ be an $\mathbb{E}$-$\mathrm{inflation}$. Then $f$ is a weak equivalence if and only if $\mathrm{Cone}(f)\in \mathcal{W}$.

\vskip5pt

$(2)$ \ Let $g$ be an $\mathbb{E}$-$\mathrm{deflation}$. Then $g$ is a weak equivalence if and only if $\mathrm{CoCone}(g)\in \mathcal{W}$.

\end{lem}

The following facts are variations involving weak equivalence, of (ET3), (ET3)$^{\rm op}$, and (ET2), respectively. 

\begin{lem}\label{lem:weak equi completion} \ $(1)$ \ Suppose that the both rows in the commutative diagram below are $\mathbb{E}$-triangles$:$
$$\xymatrix@R=0.4cm{X\ar[r]^i \ar[d]_f & Y \ar[r]^p \ar[d]_g & Z \ar@{-->}[r]^{\delta}&\\
X'\ar[r]^{i'} & Y' \ar[r]^{p'} & Z' \ar@{-->}[r]^{\delta'} & .}$$
If  $f$ and $g$ are weak equivalences, then there exists a weak equivalence $h: Z\longrightarrow Z'$ such that $hp=p'g$ and $f_*\delta=h^*\delta'$.

\vskip5pt

$(2)$ \ Suppose that the both rows in the commutative diagram below are $\mathbb{E}$-triangles$:$
$$\xymatrix@R=0.4cm{X\ar[r]^i & Y \ar[r]^p \ar[d]_g & Z \ar@{-->}[r]^{\delta} \ar[d]_h&\\
X'\ar[r]^{i'} & Y' \ar[r]^{p'} & Z' \ar@{-->}[r]^{\delta'} & .}$$
If $g$ and $h$ are weak equivalences, then there exists a weak equivalence $f: X\longrightarrow X'$ such that $gi=i'f$ and $f_*\delta=h^*\delta'$.

\vskip5pt

$(3)$ \ Suppose that the both rows in the diagram below are $\mathbb{E}$-triangles and $f_*\delta=h^*\delta':$
$$\xymatrix@R=0.4cm{X\ar[r]^i \ar[d]_f & Y \ar[r]^p & Z \ar@{-->}[r]^{\delta} \ar[d]_h&\\
X'\ar[r]^{i'} & Y' \ar[r]^{p'} & Z' \ar@{-->}[r]^{\delta'} & .}$$
If $f$ and $h$ are weak equivalences, then there exists a weak equivalence $g: Y\longrightarrow Y'$ such that $gi=i'f$ and $p'g=hp$.
\end{lem}

\begin{proof} \ \ (1) \  Since $f$ is a weak equivalence, one has $f=qj: X\overset{j}{\longrightarrow}A\overset{q}{\longrightarrow}X'$ with $j\in \mathrm{TCoFib}$ and \ $q\in \mathrm{TFib}$. By Proposition \ref{prop: pushout homotopy} there is a commutative diagram of $\mathbb{E}$-triangles:
$$\xymatrix@R=0.5cm{X \ar[r]^{i} \ar[d]_{j} & Y \ar@{.>}[d]^{j'} \ar[r]^{p} & Z \ar@{=}[d] \ar@{-->}[r]^{\delta} & \\
A \ar@{.>}[r]^{\tilde{i}} \ar[d] & B \ar@{.>}[r] \ar@{.>}[d] & Z \ar@{-->}[r]^{j_*\delta} &  \\
C \ar@{=}[r] \ar@{-->}[d] & C \ar@{-->}[d] & & \\ & & & }$$

\noindent such that the top-left square is homotopy cartesian.
Since $j\in \mathrm{TCoFib}$, one has $C\in \mathcal{C}\cap \mathcal{W}$, and hence $j'\in \mathrm{TCoFib}$. By Proposition \ref{prop: ET4 version homotopy} there is a commutative diagram of $\mathbb{E}$-triangles:
$$\xymatrix@R=0.5cm{D \ar[r] \ar@{=}[d] & A \ar[r]^{q} \ar[d]_{\tilde{i}} & X' \ar@{-->}[r] \ar@{.>}[d]^{\bar{i}} & {} \\
D \ar@{.>}[r] & B \ar@{.>}[r]^{q'} \ar[d] & E \ar@{-->}[r] \ar@{.>}[d]^{\bar{p}} & {} \\
& Z \ar@{=}[r] \ar@{-->}[d]_{j_*\delta}& Z \ar@{-->}[d]_{\varepsilon} & \\
& {}  & {} & }$$

\noindent such that $\varepsilon=q_*j_*\delta=f_*\delta$, and the top-right square is homotopy cartesian. Since $q\in \mathrm{TFib}$, one has $D\in \mathcal{F}\cap \mathcal{W}$, and hence $q'\in \mathrm{TFib}$. Put $g'=q'j'$. Then $g'\in \mathrm{Weq}$. By  Proposition \ref{prop: composition of homotopy squares} the square
$$\xymatrix@R=0.4cm{X\ar[r]^i\ar[d]_{f} & Y\ar[d]^{g'} \\ X'\ar[r]^{\bar{i}} & E }$$

\noindent is also homotopy cartesian. Thus $\xymatrix{X \ar[r]^-{\left(\begin{smallmatrix}f \\ i \end{smallmatrix}\right)} & X'\oplus Y \ar[r]^-{(\bar{i},-g')} & E \ar@{-->}[r] &}$ is an  $\mathbb{E}$-triangle. Applying $\Hom (-,Y')$ one gets an exact sequence
\[\xymatrix{\mathcal{A}(E,Y') \ar[rr]^-{\mathcal{A}((\bar{i},-g'),Y')} && \mathcal{A}(X'\oplus Y,Y') \ar[rr]^-{\mathcal{A}(\left(\begin{smallmatrix}f \\ i \end{smallmatrix}\right) ,Y')}  && \mathcal{A}(X,Y') }  \]

\noindent Since $(i',-g)\left(\begin{smallmatrix}f \\ i \end{smallmatrix}\right)=0$, there exists a morphism $g'':E \longrightarrow Y'$ such that $(i',-g)=g''\circ (\bar{i},-g')$, i.e., $g''\bar{i}=i'$ and $g''g'=g$. Thus $g''\in \mathrm{Weq}$, by Two out of three axiom.

\vskip5pt

Thus one has a factorization $g''=q''j'':E\overset{j''}{\longrightarrow}F\overset{q''}{\longrightarrow}Y'$ with $j''\in \mathrm{TCoFib}$ and $q''\in \mathrm{TFib}$. By (ET4) there is a commutative diagram of $\mathbb{E}$-triangles:
$$\xymatrix@R=0.5cm{X' \ar[r]^{\bar{i}} \ar@{=}[d] & E \ar[r]^{\bar{p}} \ar[d]_{j''} & Z \ar@{-->}[r]^{f_*\delta} \ar@{.>}[d]^{j'''} & {} \\
X' \ar@{.>}[r] & F \ar@{.>}[r] \ar[d] & H \ar@{-->}[r]^{\zeta} \ar@{.>}[d] & {} \\
& G \ar@{=}[r] \ar@{-->}[d] & G \ar@{-->}[d] & \\
& {}  & {}  & }$$
such that $(j''')^*\zeta=f_*\delta$. Since $j''\in\mathrm{TCoFib}$, one has $G\in \mathcal C\cap \mathcal W$, and hence $j'''\in\mathrm{TCoFib}$.
By Proposition \ref{prop: pullback right comp} there is a commutative diagram of $\mathbb{E}$-triangles:
$$\xymatrix@R=0.5cm{& K \ar@{=}[r] \ar[d] & K \ar@{.>}[d] & \\
X' \ar[r] \ar@{=}[d] & F \ar[r] \ar[d]_{q''} & H \ar@{-->}[r]^{\zeta} \ar@{.>}[d]^{q'''} & {} \\
X' \ar[r]^{i'} & Y' \ar[r]^{p'} \ar@{-->}[d] & Z' \ar@{-->}[r]^{\delta'} \ar@{-->}[d] & {} \\
& {} & {} & }$$

\noindent such that $(q''')^*\delta'=\zeta$. Since $q''\in\mathrm{TFib}$, one has $K\in \mathcal F\cap \mathcal W$, and hence $q'''\in\mathrm{TFib}$. Put $h=q'''j'''$. Then $h\in \mathrm{Weq}$ and one has the following commutative diagram of $\mathbb{E}$-triangles
$$\xymatrix@R=0.5cm{
X \ar[r]^{i} \ar[d]_{f} & Y \ar[r]^{p} \ar[d]_{g'} & Z \ar@{-->}[r]^{\delta} \ar@{=}[d] & {} \\
X' \ar[r]^{\bar{i}} \ar@{=}[d] & E \ar[r]^{\bar{p}} \ar[d]_{g''} & Z \ar@{-->}[r]^{f_*\delta} \ar[d]^{h} & {} \\
X' \ar[r]^{i'} & Y' \ar[r]^{p'} & Z' \ar@{-->}[r]^{\delta'} & {}}$$
where $h^*\delta'=(j''')^*(q''')^*\delta'=(j''')^*\zeta=f_*\delta$, and $hp=p'g''g'=p'g$. This completes the proof of (1).

\vskip5pt

The assertion (2) is the dual of (1), and the proof of (3) is analogous to that of (1).
\end{proof}

\subsection{Suspension and loop functors} \ Let $(\mathrm{CoFib}, \ \mathrm{Fib}, \ \mathrm{Weq})$ be an admissible model structure on a weakly idempotent complete extriangulated category $(\mathcal{A},\mathbb{E}, \mathfrak{s})$,
$\mathrm{Ho}(\mathcal{A})$ its homotopy category, and $(\mathcal{C},\mathcal{F},\mathcal{W})$ the corresponding Hovey triple.
Nakaoka and Palu \cite{NP} have defined the suspension functor $\Sigma: \mathrm{Ho}(\mathcal{A})\longrightarrow \mathrm{Ho}(\mathcal{A})$ and the loop functor $\Omega:\mathrm{Ho}(\mathcal{A})\longrightarrow\mathrm{Ho}(\mathcal{A})$, and proved that $\Sigma$ and $\Omega$ are quasi-inverse of each other. We will provide an alternative proof of this result.

\vskip5pt

By the completeness of the cotorsion pair $(\mathcal{C},\mathcal{F}\cap\mathcal{W})$, for any object $X\in \mathcal{A}$, there is an $\mathbb{E}$-triangle $X \overset{i_X}{\longrightarrow} M \overset{p_X}{\longrightarrow} \Sigma X \overset{\delta_X}{\dashrightarrow}$
with $M \in \mathcal{F}\cap\mathcal{W}$ and $\Sigma X \in \mathcal{C}$, called a {\it suspension sequence} of $X$. Also, by the completeness of the cotorsion pair $(\mathcal{C}\cap\mathcal{W},\mathcal{F})$, for any object $X$ in $\mathcal{A}$, there is an $\mathbb{E}$-triangle  $\Omega X \overset{j_X}{\longrightarrow} K \overset{q_X}{\longrightarrow} X \overset{\varepsilon_X}{\dashrightarrow}$
with $K \in \mathcal{C}\cap\mathcal{W}$ and $\Omega X \in \mathcal{F}$, called a {\it loop sequence} of $X$.

\vskip5pt
For each object $X\in \mathcal{A}$, fix a suspension sequence $X \overset{i_X}{\longrightarrow} M \overset{p_X}{\longrightarrow} \Sigma X \overset{\delta_X}{\dashrightarrow}$. For any morphism $f: X \longrightarrow Y$ in $\mathcal{A}$, applying
$\mathcal{A}(-, N)$ one gets a morphism $k: M\longrightarrow N$. Then by (ET3), there exists a morphism $g: \Sigma X \longrightarrow \Sigma Y$ in $\mathcal{A}$ such that the following diagram commutes:
$$\xymatrix@R=0.7cm{X \ar[r]^{i_X} \ar[d]_{f} & M \ar[r]^{p_X} \ar@{.>}[d]_{k} & \Sigma X \ar@{-->}[r]^{\delta_X} \ar@{.>}[d]^{g} & {} \\
Y \ar[r]^{i_Y} & N \ar[r]^{p_Y} & \Sigma X \ar@{-->}[r]^{\delta_Y} & {}}$$ and $f_*\delta_X =g^*\delta_Y$. Define $\Sigma f := \gamma(g)\in \mathrm{Ho}(\mathcal{A})$.
Then $\Sigma: \mathcal{A} \longrightarrow \mathrm{Ho}(\mathcal{A})$ is a well-defined functor (\cite[6.2]{NP}) and $\Sigma$ is additive. A representative $g$ of $\Sigma f$ in $\mathcal{A}$ is also denoted by $\Sigma f$.
\vskip5pt

Since $M, N \in \mathcal{W}$, any morphism $k: M \longrightarrow N$ is a weak equivalence. If $f$ is a weak equivalence, then the above $g$ can be chosen to be a weak equivalence $g$ such that $\Sigma f=\gamma(g)$, by Lemma \ref{lem:weak equi completion}. Thus, $\Sigma$ sends weak equivalences in $\mathcal{A}$ to isomorphisms in $\mathrm{Ho}(\mathcal{A})$. By the universal property of the localization functor, $\Sigma: \mathcal{A} \longrightarrow \mathrm{Ho}(\mathcal{A})$ induces a functor $\bar\Sigma: \mathrm{Ho}(\mathcal{A}) \longrightarrow \mathrm{Ho}(\mathcal{A})$ such that $\Sigma = \bar\Sigma \circ \gamma$.

\vskip5pt

For objects $X, Y$ in $\mathcal A$, let $(X\oplus Y,\ e^X,\ e^Y)$ be the coproduct of $X$ and $Y$ in $\mathcal A$. Since the localization functor $\gamma:\mathcal A \longrightarrow \mathrm{Ho}(\mathcal A)$ is additive (c.f. Lemma \ref{addlocal}), $(X\oplus Y,\ \gamma(e^X),\ \gamma(e^Y))$ is the coproduct of $ X$ and $Y$ in $\mathrm{Ho}(\mathcal A)$.
Since  $\Sigma:\mathcal A \longrightarrow \mathrm{Ho}(\mathcal A)$ is an additive functor, $(\Sigma(X\oplus Y),\ \Sigma e^X,\ \Sigma e^Y)$ is the coproduct of $\Sigma X$ and $\Sigma Y$ in $\mathrm{Ho}(\mathcal A)$.
Thus, the functor $\bar{\Sigma}:\mathrm{Ho}(\mathcal A)\longrightarrow\mathrm{Ho}(\mathcal A)$ preserves coproducts in $\mathrm{Ho}(\mathcal A)$, and hence $\bar{\Sigma}$ is  additive, which  is called the {\it suspension functor}, and will be simply denoted by $\Sigma$.

\vskip5pt

Dually, for any object $X\in \mathcal{A}$, fix a loop sequence $\Omega X \overset{j_X}{\longrightarrow} K \overset{q_X}{\longrightarrow} X \overset{\varepsilon_X}{\dashrightarrow}$.
For any morphism $f: X \longrightarrow Y$ in $\mathcal{A}$, there exists a morphism $h: \Omega X \longrightarrow \Omega Y$ in $\mathcal{A}$ such that $h_* \varepsilon_X = f^* \varepsilon_Y$. Define $\Omega f := \gamma(h) \in \mathrm{Ho}(\mathcal{A})$. Then $\Omega: \mathcal{A} \longrightarrow \mathrm{Ho}(\mathcal{A})$ is an additive functor. A representative $h$ of $\Omega f$ in $\mathcal{A}$ will be also denoted by $\Omega f$.
Similarly, $\Omega$ induces an additive functor $\bar\Omega: \mathrm{Ho}(\mathcal{A}) \longrightarrow \mathrm{Ho}(\mathcal{A})$ such that $\Omega = \bar\Omega \circ \gamma$, which is called the {\it loop functor}, and will be simply denoted by $\Omega$.

\begin{prop}[{\rm \cite[Proposition 6.14]{NP}}]  \label{shift} \  Let $\Sigma, \Omega: \mathrm{Ho}(\mathcal{A}) \longrightarrow \mathrm{Ho}(\mathcal{A})$ be the suspension functor and the loop functor, respectively.
Then  $\mathrm{Id} \cong \Sigma\Omega $ and $ \mathrm{Id} \cong \Omega\Sigma$.
\end{prop}

\begin{proof} \ Here we provide an alternative proof. The definition of the natural isomorphism $\eta:{\rm Id}\longrightarrow \Sigma \Omega$ is the same as \cite{NP}.
For object $X \in \mathcal{A}$, one has the suspension sequence $X \overset{i_X}{\longrightarrow} M \overset{p_X}{\longrightarrow} \Sigma X \overset{\delta_X}{\dashrightarrow}$ and the loop sequence $\Omega\Sigma X \overset{j_X}{\longrightarrow} K \overset{q_X}{\longrightarrow} \Sigma X \overset{\varepsilon_X}{\dashrightarrow}$.
By Proposition \ref{prop: pullback 1} there is a commutative diagram of $\mathbb{E}$-triangles:
$$\xymatrix@R=0.8cm{& \Omega\Sigma X \ar@{=}[r] \ar@{.>}[d]_{l_X} & \Omega\Sigma X \ar[d]^{j_X} & \\
X \ar@{.>}[r]^{k_X} \ar@{=}[d] & P \ar@{.>}[r]^{r_X} \ar@{.>}[d]_{t_X} & K \ar@{-->}[r]^{\lambda_X} \ar[d]^{q_X} & {} \\
X \ar[r]^{i_X} & M \ar[r]^{p_X} \ar@{-->}[d]_{\mu_X} & \Sigma X \ar@{-->}[r]^{\delta_X} \ar@{-->}[d]^{\varepsilon_X} & {} \\
& {}  & {}  & }$$
such that $(q_X)^*\delta_X=\lambda_X, \ \ (p_X)^*\varepsilon_X=\mu_X, \ \ (k_X)_*\delta_X+(l_X)_*\varepsilon_X=0.$ Since $K \in \mathcal{C} \cap \mathcal{W}$, one has $k_X \in \mathrm{TCoFib}$. Since $l_X$ is an $\mathbb{E}$-inflation and $\mathrm{Cone}(l_X) \in \mathcal{W}$, it follows from Lemma \ref{lem: weq} that $l_X$ is a weak equivalence, and hence $\eta_{_X} := \gamma(l_X)^{-1} \gamma(k_X) : X \longrightarrow \Omega\Sigma X$ is an isomorphism in $\mathrm{Ho}(\mathcal{A})$.

\vskip5pt

It remains to show that $\eta_{_X}: X\longrightarrow \Omega \Sigma X$ is functorial in $X$. Nakaoka and Palu \cite{NP} have proved this, via the properties of ``connecting morphisms'' they introduced.
Here we give a direct proof. For this, we first show that, for any morphism $f: X \longrightarrow Y$ in $\mathcal{A}$, the following square commutes in $\mathrm{Ho}(\mathcal{A})$:
\begin{equation}\label{eq:commu_narural2}
\xymatrix@R=0.6cm{X \ar[r]^{\eta_{_X}} \ar[d]_{\gamma(f)} & \Omega \Sigma X \ar[d]^{\Omega \Sigma \gamma(f)} \\ Y \ar[r]^{\eta_{_Y}} & \Omega \Sigma Y.}
\end{equation}

\noindent By Proposition \ref{prop: pullback 1} and the definitions of $\Sigma$ and $\Omega$, there is a commutative diagram of $\mathbb{E}$-triangles in $\mathcal{A}$:
\[\xymatrix@R=0.5cm@C=0.7cm{& & \Omega\Sigma X \ar@{=}[rr] \ar[dd]_(0.6){l_X} \ar[dr]^{\Omega\Sigma f} & & \Omega\Sigma X \ar[dr]^{\Omega\Sigma f} \ar[dd]_(0.3){j_X} & & \\
& & & \Omega\Sigma Y \ar@{=}[rr] \ar[dd]^(0.6){l_Y} & & \Omega\Sigma Y \ar[dd]^(0.4){j_Y} & \\
X \ar[rr]^(0.4){k_X} \ar@{=}[dd] \ar[dr]^{f} & & P \ar[rr]^(0.3){r_X} \ar[dd] & & K \ar@{-->}[r]^{\lambda_X} \ar[dr]^{\bar{f}} \ar[dd]_(0.3){q_X} & {} & \\
& Y \ar[rr]^(0.3){k_Y} \ar@{=}[dd] & & Q \ar[rr]^(0.3){r_Y} \ar[dd] & & L \ar@{-->}[r]^{\lambda_Y} \ar[dd]^{q_Y} & {} \\
X \ar[rr]^(0.4){i_X} \ar[dr]^{f} & & M \ar[rr]^(0.4){p_X} \ar@{-->}[dd]_>>{\mu_X} \ar[dr]^{\tilde{f}} & & \Sigma X \ar@{-->}[r]^{\delta_X} \ar@{-->}[dd]^>>{\varepsilon_X} \ar[dr]^{\Sigma f} & {} & \\
& Y \ar[rr]^(0.3){i_Y} & & N \ar[rr]^(0.4){p_Y} \ar@{-->}[dd]_>>{\mu_Y} & & \Sigma Y \ar@{-->}[r]^{\delta_Y} \ar@{-->}[dd]^>>{\varepsilon_Y} & {} \\
& & & {}  & & {}  & \\ & & & {} & & {} & }
\]
such that
$$\begin{matrix}
& f_*\delta_X=(\Sigma f)^*\delta_Y,  \ \ & (q_X)^*\delta_X =\lambda_X, \ \ &(p_X)^*\varepsilon_X=\mu_X,  \ \ & (k_X)_*\delta_X+(l_X)_*\varepsilon_X =0, \\
& (\Sigma f)^* \varepsilon_Y=(\Omega \Sigma f)_*\varepsilon_X, \ \ &  (q_Y)^*\delta_Y  =\lambda_Y,  \ \ & (p_Y)^*\varepsilon_Y  =\mu_Y, \ \ & (k_Y)_*\delta_Y+(l_Y)_*\varepsilon_Y =0.
\end{matrix}$$
Since
\begin{align*}
f_*\lambda_X &=f_*(q_X)^*\delta_X=(q_X)^*f_*\delta_X\\
&=(q_X)^*(\Sigma f)^*\delta_Y=((\Sigma f )q_X)^*\delta_Y\\
&=(q_Y \bar{f})^*\delta_Y=\bar{f}^*(q_Y)^*\delta_Y\\
&=\bar{f}^*\lambda_Y,
\end{align*}

\noindent it follows from (ET2) that there is a morphism $s:P \longrightarrow Q$ such that $sk_X=k_Y f$ and $r_Ys=\bar{f}r_X$. Since
$$\begin{aligned}
(sl_X-l_Y(\Omega\Sigma f))_*\varepsilon_X &=s_*(l_X)_*\varepsilon_X-(l_Y)_*(\Omega \Sigma f)_*\varepsilon_X=-s_*(k_X)_*\delta_X-(l_Y)_*(\Sigma f)^* \varepsilon_Y\\ &=-(sk_X)_*\delta_X-(\Sigma f)^*(l_Y)_*\varepsilon_Y=-(sk_X)_*\delta_X+(\Sigma f)^*(k_Y)_*\delta_Y\\ &=-(sk_X)_*\delta_X+(k_Y)_*(\Sigma f)^*\delta_Y=-(sk_X)_*\delta_X+(k_Y)_*f_*\delta_X\\ &=(-sk_X+k_Yf)_*\delta_X=0_*\delta_X=0,
\end{aligned}$$
it follows from Corollary \ref{cor: factorization} that $sl_X - l_Y (\Omega \Sigma f)$ factors through $K\in \mathcal{C}\cap \mathcal{W}$, and hence $\gamma(sl_X)=\gamma(l_Y (\Omega \Sigma f))$ in $\mathrm{Ho}(\mathcal{A})$.
Thus one has commutative diagram in $\mathrm{Ho}(\mathcal{A})$:
$$\xymatrix@R=0.6cm{X \ar[r]^{\gamma(k_X)} \ar[d]_{\gamma(f)} & P \ar[d]^{\gamma(s)} & \Omega\Sigma X \ar[d]^{\Omega\Sigma\gamma(f)} \ar[l]_{\gamma(l_X)} \\
Y \ar[r]_{\gamma(k_Y)} & Q & \Omega\Sigma Y. \ar[l]^{\gamma(l_Y)}}$$

\noindent From this one sees that square \eqref{eq:commu_narural2} commutes in $\mathrm{Ho}(\mathcal{A})$. In particular, if $f: X \longrightarrow Y$ is a weak equivalence, then the following square also commutes in $\mathrm{Ho}(\mathcal{A})$:
$$\xymatrix@R=0.6cm{Y\ar[r]^{\eta_{_Y}} \ar[d]_{\gamma(f)^{-1}} & \Omega\Sigma Y \ar[d]^{\Omega\Sigma\gamma(f)^{-1}} \\
X \ar[r]^-{\eta_{_X}} & \Omega\Sigma X.}$$

For any morphism $g: X \longrightarrow Y$ in $\mathrm{Ho}(\mathcal{A})$, by Proposition \ref{prop:canonical} there exists a morphism $f$ in $\mathcal{A}$ such that the following square commutes in $\mathrm{Ho}(\mathcal{A})$:
$$\xymatrix@R=0.6cm{X \ar[r]^{g} \ar[d]_{\theta_X} & Y \ar[d]^{\theta_Y} \\ RQX \ar[r]^{\gamma(f)} & RQY}$$
where $\theta_X=\gamma(\sigma_X)\gamma(\rho_X)^{-1},\ \theta_Y=\gamma(\sigma_Y)\gamma(\rho_Y)^{-1}$. Thus $g=\gamma(\rho_Y)\gamma(\sigma_Y)^{-1}\gamma(f)\gamma(\sigma_X)\gamma(\rho_X)^{-1}$.
Combining the following commutative squares in $\mathrm{Ho}(\mathcal{A}):$
\[
\xymatrix@C=1.3cm{X \ar[r]^-{\gamma(\rho_X)^{-1}} \ar[d]_{\eta_{_{X}}}& QX \ar[r]^-{\gamma(\sigma_X)} \ar[d]_{\eta_{_{QX}}}& RQX \ar[r]^-{\gamma(f)} \ar[d]_{\eta_{_{RQX}}} & RQY \ar[r]^-{\gamma(\sigma_Y)^{-1}} \ar[d]_{\eta_{_{RQY}}} & QY \ar[r]^-{\gamma(\rho_Y)} \ar[d]_{\eta_{_{QY}}} & Y \ar[d]_{\eta_{_{Y}}}\\
\Omega\Sigma X \ar[r]^-{\Omega\Sigma\gamma(\rho_X)^{-1}} & \Omega\Sigma QX \ar[r]^-{\Omega\Sigma\gamma(\sigma_X)} & \Omega\Sigma RQX \ar[r]^-{\Omega\Sigma\gamma(f)} & \Omega\Sigma RQY \ar[r]^-{\Omega\Sigma \gamma(\sigma_Y)^{-1}} & \Omega\Sigma  QY \ar[r]^{\Omega\Sigma\gamma(\rho_Y)} & \Omega\Sigma Y}
\]
one has the commutative square:
$$\xymatrix@R=0.6cm{X \ar[r]^{g} \ar[d]_{\eta_{_X}}   & Y \ar[d]^{\eta_{_Y}} \\ \Omega \Sigma X  \ar[r]^{\Omega\Sigma g} & \Omega \Sigma Y.}$$ That is, $\eta=(\eta_{_X})$ is a natural isomorphism.

\vskip5pt

Similarly, there exists a natural isomorphism $\zeta: {\rm Id}\longrightarrow \Omega \Sigma$.
\end{proof}

\subsection{Triangulated structure on $\mathrm{Ho}(\mathcal{A})$} \ Let $(\mathrm{CoFib},\ \mathrm{Fib},\ \mathrm{Weq})$ be an admissible model structure on a weakly idempotent complete extriangulated category $(\mathcal{A}, \mathbb{E}, \mathfrak{s})$,  $(\mathcal{C},\ \mathcal{F},\ \mathcal{W})$ the corresponding Hovey triple, and $\Sigma$ the suspension functor. We will define the class $\Delta$ of distinguished triangles in $\mathrm{Ho}(\mathcal{A})$.

\begin{defn}\label{def: disinguishedtriangle} \ Let $f:X \longrightarrow Y$ be a morphism in $\mathcal{A} $, and $X\overset{i_X}{\longrightarrow}M\overset{p_X}{\longrightarrow}\Sigma X\overset{\delta_X}{\dashrightarrow}$ the suspension sequence of $X$. Then there is a commutative diagram of $ \mathbb{E} $-triangles in $ \mathcal{A} $:
$$\xymatrix@R=0.5cm{X \ar[r]^-{i_X} \ar[d]_{f} & M \ar[r]^-{p_X} \ar[d] & \Sigma X \ar@{-->}[r]^-{\delta_X} \ar@{=}[d] & {} \\
Y \ar[r]^{g} & C(f) \ar[r]^{h} & \Sigma X \ar@{-->}[r]^-{f_*\delta_X} & {}}$$

\noindent The triangle in $\mathrm{Ho}(\mathcal{A})$
\[
\xymatrix{X \ar[r]^-{\gamma(f)} & Y \ar[r]^-{\gamma(g)} & C(f) \ar[r]^-{\gamma(h)} & \Sigma X}
\]
is called the {\it standard triangle} induced by $f$. Denote $\Delta$ by the class of triangles isomorphic to standard triangles in $\mathrm{Ho}(\mathcal{A})$. 
\end{defn}

The distinguished triangles here are similar to the classical construction in \cite{Hap1, Hap2, Heller}, but differ from the one in \cite{NP}.
The standard triangle $X \overset{\gamma(f)}{\longrightarrow} Y \overset{\gamma(g)}{\longrightarrow} Z \overset{l(\delta)}{\longrightarrow} \Sigma X$ in $\mathrm{Ho}(\mathcal{A})$ in \cite{NP} is the image under $\gamma$ of an $\mathbb{E}$-triangle $X \longrightarrow Y \longrightarrow Z \overset{\delta}{\dashrightarrow}$ in $\mathcal{A}$ together with a ``connecting morphism'' $l(\delta): Z \longrightarrow \Sigma X$ in $\mathrm{Ho}(\mathcal{A})$.

\subsection{Main results}
\begin{thm}\label{thm: main} \ Let $ (\mathrm{CoFib},\mathrm{Fib},\mathrm{Weq}) $ be an admissible model structure on a weakly idempotent complete extriangulated category $ (\mathcal{A}, \mathbb{E},\mathfrak{s})$, 
$\Sigma: \mathrm{Ho}(\mathcal{A})\longrightarrow \mathrm{Ho}(\mathcal{A})$ the suspension functor, and $\Delta$ the class of triangles in Definition $\ref{def: disinguishedtriangle}$. 
Then $ (\mathrm{Ho}(\mathcal{A}), \Sigma, \Delta) $ is a triangulated category.
\end{thm}

To prove Theorem \ref{thm: main}, by Proposition \ref{shift} it remains to verify the axioms (TR1), (TR2), (TR3) and (TR4).

\vskip5pt

For any object $X\in\mathrm{Ho}(\mathcal{A})$, consider the following commutative diagram
$$\xymatrix@R=0.5cm{X \ar[r]^{i_X} \ar[d]_{1_X} & M \ar[r]^{p_X} \ar[d]^{1_M} & \Sigma X \ar@{-->}[r]^-{\delta_X} \ar@{=}[d] & {} \\ X \ar[r]^{i_X} & M \ar[r]^{p_X} & \Sigma X \ar@{-->}[r]^-{\delta_X} & {}}$$

\noindent where $X\overset{i_X}{\longrightarrow}M\overset{p_X}{\longrightarrow}\Sigma X\overset{\delta_X}{\dashrightarrow}$ is the suspension sequence of $X$. By definition $X \stackrel {{\rm Id}} \longrightarrow X \longrightarrow M \longrightarrow \Sigma X $ is a standard triangle. Since $M \in \mathcal{F} \cap \mathcal{W}$, $M\cong 0$ in $ \mathrm{Ho}(\mathcal{A})$. Thus, the triangle $ X \stackrel {{\rm Id}} \longrightarrow X \longrightarrow 0 \longrightarrow \Sigma X $ is in $ \Delta $.

\vskip5pt

For any morphism $ f: X \longrightarrow Y $ in $ \mathrm{Ho}(\mathcal{A}) $, by Proposition \ref{prop:canonical} there is a morphism $ \tilde{f}: RQX \longrightarrow RQY $ in $ \mathcal{A} $ such that the following square commutes in $\mathrm{Ho}(\mathcal{A}) $:
$$\xymatrix@R=0.4cm{X \ar[r]^f \ar[d]_{\theta_X} & Y \ar[d]^{\theta_Y} \\ RQX \ar[r]^{\gamma(\tilde{f})} & RQY}$$

\noindent where $\theta_X$ and $\theta_Y$ are isomorphisms in $\mathrm{Ho}(\mathcal{A})$. Let $\xymatrix{RQX \ar[r]^-{\gamma(\tilde{f})} & RQY \ar[r]^-{\gamma(\tilde{g})} & C(\tilde{f}) \ar[r]^-{\gamma(\tilde{h})} & \Sigma RQX}$
be the standard triangle induced by $ \tilde{f} $. Then there is an isomorphism of triangles in $\mathrm{Ho}(\mathcal{A})$:
$$\xymatrix@R=0.5cm{X \ar[r]^f \ar[d]_{\theta_X} & Y \ar[r]^{\gamma(\tilde{g})\circ\theta_Y} \ar[d]^{\theta_Y} & C(f) \ar[r]^{\Sigma\theta_X^{-1}\circ\gamma(\tilde{h})} \ar@{=}[d] & \Sigma X \ar[d]^{\Sigma\theta_X} \\
RQX \ar[r]^-{\gamma(\tilde{f})} & RQY \ar[r]^-{\gamma(\tilde{g})} & C(f) \ar[r]^-{\gamma(\tilde{h})} & \Sigma RQX.}$$

\noindent Thus, the first row belongs to $\Delta$, which shows that every morphism $f$ in  $\mathrm{Ho}(\mathcal{A})$ can be embedded into a triangle in $\Delta$. This completes the proof of (TR1).

\vskip5pt

To prove (TR2),  it suffices to verify the following lemma.
\begin{lem}\label{lem: suspension} \ For a  standard triangle  $\xymatrix{X \ar[r]^-{\gamma(f)} & Y \ar[r]^-{\gamma(g)} & C(f) \ar[r]^-{\gamma(h)} & \Sigma X}$ in $\mathrm{Ho}(\mathcal{A})$, 
the triangle $\xymatrix{Y \ar[r]^-{\gamma(g)} & C(f) \ar[r]^-{\gamma(h)} & \Sigma X \ar[r]^-{-\Sigma\gamma(f)} & \Sigma Y}$ is in $\Delta$.
\end{lem}

\begin{proof} \ By definition there is a commutative diagram of $\mathbb{E}$-triangles in $\mathcal{A}$:
$$\xymatrix@R=0.5cm{X \ar[r]^{i_X} \ar[d]_{f} & M \ar[r]^{p_X} \ar[d] & \Sigma X \ar@{-->}[r]^-{\delta_X} \ar@{=}[d] & {} \\
Y \ar[r]^{g} & C(f) \ar[r]^{h} & \Sigma X \ar@{-->}[r]^-{f_*\delta_X} & {.}}$$

\noindent By the definition of $\Sigma$ one has a commutative diagram of $\mathbb{E}$-triangles in $\mathcal{A}$:
$$\xymatrix@R=0.5cm{X \ar[r]^{i_X} \ar[d]_f & M \ar[r]^{p_X} \ar[d] & \Sigma X \ar@{-->}[r]^-{\delta_X} \ar[d]_{f'} & {} \\ Y \ar[r]^-{i_Y} & N \ar[r]^-{p_Y} & \Sigma Y \ar@{-->}[r]^-{\delta_Y} & {}}$$
such that $f_*\delta_X=f'^*\delta_Y$, where $\Sigma\gamma(f)=\gamma(f')$. By Proposition \ref{prop: pullback mid complete homotopy square} there is a morphism $t: C(f)\longrightarrow N$ such that the following diagram
$$\xymatrix@R=0.7cm{Y \ar[r]^{g} \ar@{=}[d] & C(f) \ar[r]^{h} \ar@{.>}[d]_{t} & \Sigma X \ar@{-->}[r]^-{f_*\delta_X} \ar[d]_{f'} & {} \\
Y \ar[r]^-{i_Y} & N \ar[r]^-{p_Y} & \Sigma Y \ar@{-->}[r]^-{\delta_Y} & {}}$$
 commutes in $\mathcal{A}$, and the right-hand square is homotopy cartesian, that is, $$\xymatrix@C=0.6cm{C(f) \ar[rr]^-{\left( \begin{smallmatrix} h \\ t \end{smallmatrix} \right)} && \Sigma X\oplus N \ar[rr]^-{(-f',\ p_Y )} && \Sigma Y \ar@{-->}[r]^-{g_*\delta_Y}&}$$
is an $\mathbb{E}$-triangle. Thus one has a commutative diagram of $\mathbb{E}$-triangles 
$$\xymatrix{Y \ar[r]^{i_Y} \ar[d]_{g}& N \ar[d]_{ \left( \begin{smallmatrix} 0 \\ 1 \end{smallmatrix} \right)}  \ar[rr]^{p_Y} && \Sigma Y \ar@{=}[d] \ar@{-->}[r]^{\delta_Y} &   
\\ C(f) \ar[r]^-{\left( \begin{smallmatrix} h \\ t \end{smallmatrix} \right)} & \Sigma X\oplus N \ar[rr]^-{(-f',\ p_Y )} && \Sigma Y \ar@{-->}[r]^-{g_*\delta_Y}&} $$
By definition one has the standard triangle induced by $g$:
$$\xymatrix@C=0.8cm{Y \ar[rr]^{\gamma(g)} && C(f) \ar[rr]^-{\gamma\left( \begin{smallmatrix} h \\ t \end{smallmatrix} \right)} && \Sigma X\oplus N \ar[rr]^-{\gamma(-f', p_Y )} && \Sigma Y }$$

\noindent Since $N\in \mathcal{W}$, $N$ is the zero object in $\mathrm{Ho}(\mathcal{A})$. Therefore 
$\xymatrix{Y \ar[r]^-{\gamma(g)} & C(f) \ar[r]^-{\gamma(h)} & \Sigma X \ar[r]^-{-\Sigma\gamma(f)} & \Sigma Y}$ is a triangle in $\Delta$. \end{proof}

\vskip5pt

To prove (TR3), we first show a lemma, which makes two terms in a standard triangle lie in $\mathcal{C}\cap\mathcal{F}$.

\begin{lem} \label{lem:TR3} \ For any standard triangle $\xymatrix{X \ar[r]^-{\gamma(f)} & Y \ar[r]^-{\gamma(g)} & C(f) \ar[r]^-{\gamma(h)} & \Sigma X}$ in $\mathrm{Ho}(\mathcal{A})$, 
there is a morphism $\tilde{f}:RQX\longrightarrow RQY$ in $\mathcal{A}$ and an isomorphism of triangles in $\mathrm{Ho}(\mathcal{A}):$
$$\xymatrix@R=0.5cm{X \ar[r]^{\gamma(f)} \ar[d]_{\theta_X} & Y \ar[r]^{\gamma(g)} \ar[d]^{\theta_Y} & C(f) \ar[r]^{\gamma(h)} \ar[d]^{w} & \Sigma X \ar[d]^{\Sigma\theta_X} \\
RQX \ar[r]^{\gamma(\tilde{f})} & RQY \ar[r]^{\gamma(\tilde{g})} & C(\tilde{f}) \ar[r]^{\gamma(\tilde{h})} & \Sigma RQX}$$

\noindent where the second row is the standard triangle induced by $\tilde{f}$, and  $RQX$ and $RQY$ are in $\mathcal{C}\cap\mathcal{F}$.
\end{lem}
\begin{proof} \ Recall that, by the completeness of cotorsion pairs $(\mathcal{C}, \mathcal{F}\cap \mathcal{W})$  and $(\mathcal{C}\cap \mathcal{W}, \mathcal{F})$,  one has  
$\mathbb{E}$-triangles in $\mathcal A$: 
\begin{align*} & R_X \longrightarrow QX \overset{\rho_X}{\longrightarrow}X\dashrightarrow, \ \ \ \ \ \ QX \overset{\sigma_X}{\longrightarrow} RQX \longrightarrow Q_X\dashrightarrow \\ 
& R_Y \longrightarrow QY \overset{\rho_Y }{\longrightarrow}Y \dashrightarrow, \ \ \ \ \ \ QY  \overset{\sigma_Y }{\longrightarrow} RQY  \longrightarrow Q_Y \dashrightarrow \end{align*}
where $QX, QY\in \mathcal{C}$, \ $R_X, R_Y\in \mathcal{F}\cap \mathcal{W}$, \  $RQX, RQY\in \mathcal{F}$, \  $Q_X, Q_Y\in \mathcal{C}\cap \mathcal{W}$. Thus $\rho_X,\ \sigma_X, \rho_Y, \sigma_Y$ are weak equivalences, and one has canonical isomorphisms in $\mathrm{Ho}(\mathcal{A}):$
$$\theta_X:=\gamma(\sigma_X)\circ\gamma(\rho_X)^{-1}:X\longrightarrow RQX, \ \ \ \ \ \theta_Y:=\gamma(\sigma_Y)\circ\gamma(\rho_Y)^{-1}: Y\longrightarrow RQY$$

Applying $\mathcal{A}(QX,-)$ to $R_Y \longrightarrow QY \overset{\rho_Y }{\longrightarrow}Y \dashrightarrow$, one gets an exact sequence
$$\xymatrix{\mathcal{A}(QX,QY) \ar[rr]^{\mathcal{A}(QX,\rho_Y)}  && \mathcal{A}(QX,Y)  \ar[rr] && \mathbb{E}(QX,R_Y) }   $$

\noindent Since $\mathbb{E}(QX,R_Y)=0$, there is a morphism $\bar{f}: QX \longrightarrow QY$ such that $f \rho_X=\rho_Y \bar{f}$. 
Applying $\mathcal{A}(-, RQY)$ to $QX \overset{\sigma_X}{\longrightarrow} RQX \longrightarrow Q_X\dashrightarrow$, one gets an exact sequence
$$\xymatrix{\mathcal{A}(RQX,RQY) \ar[rr]^{\mathcal{A}(\sigma_X,RQY)}  && \mathcal{A}(QX,RQY)  \ar[rr] && \mathbb{E}(Q_X,RQY) }   $$

\noindent Since $\mathbb{E}(Q_X,RQY)=0$, there is a morphism $\tilde{f}: RQX \longrightarrow RQY$ such that $\tilde{f}\sigma_X=\sigma_Y \bar{f}$. 
\vskip5pt

Consider the suspension sequences of $X$ and $QX$: 
$$\xymatrix{X \ar[r] & M_X \ar[r] & \Sigma X \ar@{-->}[r]^-{\delta_X} & {} },  \ \  \xymatrix{QX \ar[r] & M_{QX} \ar[r] & \Sigma QX \ar@{-->}[r]^-{\delta_{QX}} & {.} }$$

\noindent By the definition of $\Sigma \rho_X$, one has the commutative diagram of $\mathbb{E}$-triangles in $\mathcal{A}$:
$$\xymatrix@R=0.5cm{X \ar[r] \ar[d]_{\rho_X} & M_{X} \ar[r] \ar[d] & \Sigma X \ar@{-->}[r]^-{\delta_{X}} \ar[d]^{\Sigma \rho_X} & {} \\
QX \ar[r] & M_{QX} \ar[r] & \Sigma QX \ar@{-->}[r]^-{\delta_{QX}} & {}}$$
such that $(\rho_X)_*\delta_X=(\Sigma \rho_X)^*\delta_{QX}$.
Also, there are commutative diagrams of $\mathbb{E}$-triangles in $\mathcal{A}$:

\begin{minipage}[h]{0.45\linewidth}
\centering
$$\xymatrix@R=0.5cm{X \ar[r] \ar[d]_{f} & M_X \ar[r] \ar[d] & \Sigma X \ar@{-->}[r]^-{\delta_X} \ar@{=}[d] & {} \\
Y \ar[r]^{g} & C(f) \ar[r]^{h} & \Sigma X \ar@{-->}[r]^-{f_*\delta_X} & {}}$$
\end{minipage}
\hfill
\begin{minipage}[h]{0.45\linewidth}
\centering
$$\xymatrix@R=0.5cm{QX \ar[r] \ar[d]_{\bar{f}} & M_{QX} \ar[r] \ar[d] & \Sigma QX \ar@{-->}[r]^-{\delta_{QX}} \ar@{=}[d] & {} \\
QY \ar[r]^{\bar{g}} & C(\bar{f}) \ar[r]^{\bar{h}} & \Sigma QX \ar@{-->}[r]^-{\bar{f}_*\delta_{QX}} & {.}}$$
\end{minipage}

\noindent Consider the diagram of $\mathbb{E}$-triangles in $\mathcal A$:
$$\xymatrix@R=0.5cm{QY \ar[r]^{\bar{g}} \ar[d]_{\rho_Y} & C(\bar{f}) \ar[r]^{\bar{h}} \ar@{.>}[d]_{\varphi} & \Sigma QX \ar@{-->}[r]^-{\bar{f}_*\delta_{QX}} \ar[d]^{\Sigma \rho_X} & {} \\ Y \ar[r]^g & C(f) \ar[r]^h & \Sigma X \ar@{-->}[r]^-{f_*\delta_X} & {} }  $$
Since $$(\rho_Y)_*\bar{f}_*\delta_{QX} = (\rho_Y\bar{f})_*\delta_{QX} = (f \rho_X)_* \delta_{QX}=f_*(\rho_X)_*\delta_{QX}=f_*(\Sigma \rho_X)^*\delta_X=(\Sigma \rho_X)^*f_*\delta_X$$
\noindent by Lemma \ref{lem:weak equi completion} (3) there is a weak equivalence $\varphi$ such that the diagram above commutes. 

Consider the suspension sequence $\xymatrix{RQX \ar[r] & M_{RQX} \ar[r] & \Sigma RQX \ar@{-->}[r]^-{\delta_{RQX}} & {} }$ of $RQX$. By the definition of $\Sigma \sigma_X$, one has the commutative diagram of $\mathbb{E}$-triangles in $\mathcal{A}$:
$$\xymatrix@R=0.5cm{QX \ar[r] \ar[d]_{\sigma_X} & M_{QX} \ar[r] \ar[d] & \Sigma QX \ar@{-->}[r]^-{\delta_{QX}} \ar[d]^{\Sigma \sigma_X} & {} \\
RQX \ar[r] & M_{RQX} \ar[r] & \Sigma RQX \ar@{-->}[r]^-{\delta_{RQX}} & {}}$$
such that $(\Sigma \sigma_X)^*\delta_{RQX} = (\sigma_X)_*\delta_{QX}$. 
Also, there is a commutative diagram of $\mathbb{E}$-triangles in $\mathcal{A}$:
$$\xymatrix@R=0.5cm{RQX \ar[r] \ar[d]_{\tilde{f}} & M_{RQX} \ar[r] \ar[d] & \Sigma RQX \ar@{-->}[r]^-{\delta_{RQX}} \ar@{=}[d] & {} \\
RQY \ar[r]^{\tilde{g}} & C(\tilde{f}) \ar[r]^{\tilde{h}} & \Sigma RQX \ar@{-->}[r]^-{\tilde{f}_*\delta_{RQX}} & {.}}$$
\noindent Consider the following diagram of $\mathbb{E}$-triangles in $\mathcal A$:
$$\xymatrix@R=0.5cm{QY \ar[r]^{\bar{g}} \ar[d]_{\sigma_Y} & C(\bar{f}) \ar[r]^{\bar{h}} \ar@{.>}[d]_{\psi} & \Sigma QX \ar@{-->}[r]^-{\bar{f}_*\delta_{QX}} \ar[d]^{\Sigma \sigma_X} & {} \\ RQY \ar[r]^{\tilde{g}} & C(\tilde{f}) \ar[r]^{\tilde{h}} & \Sigma RQX \ar@{-->}[r]^-{\tilde{f}_*\delta_{RQX}} & {} }  $$

\noindent Since $$(\sigma_Y)_*\bar{f}_*\delta_{QX}= (\sigma_Y\bar{f})_*\delta_{QX} = (\tilde{f}\sigma_X)_*\delta_{QX}  = \tilde{f}_*(\sigma_X)_*\delta_{QX}=\tilde{f}_*(\Sigma \sigma_X)^*\delta_{RQX}=(\Sigma \sigma_X)^*\tilde{f}_*\delta_{RQX}$$ 
by Lemma \ref{lem:weak equi completion} (3) there is a weak equivalence $\psi$ such that the diagram above commutes.

\vskip5pt

Therefore one has the commutative diagram in $\mathcal{A}$:
\begin{equation}\label{eq: phipsi}
\begin{aligned}
\xymatrix@R=0.7cm{X \ar[r]^{f} & Y \ar[r]^{g} & C(f) \ar[r]^{h}& \Sigma X \\
QX \ar[u]^{\rho_X} \ar[r]^{\bar{f}} \ar[d]_{\sigma_X} & QY \ar[u]^{\rho_Y} \ar[r]^{\bar{g}} \ar[d]_{\sigma_Y}  & C(\bar{f}) \ar@{.>}[u]_{\varphi} \ar[r]^{\bar{h}} \ar@{.>}[d]^{\psi} & \Sigma QX \ar[u]_{\Sigma\rho_X} \ar[d]^{\Sigma\sigma_X} \\
RQX \ar[r]^{\tilde{f}} & RQY \ar[r]^{\tilde{g}} & C(\tilde{f}) \ar[r]^{\tilde{h}} & \Sigma RQX}
\end{aligned}
\end{equation}

\vskip5pt

\noindent where the images under $\gamma$ of the first, second, and third rows are standard triangles induced by  $f$, $\bar{f}$, and $\tilde{f}$, respectively. Since $\theta_X=\gamma(\sigma_X)\circ\gamma(\rho_X)^{-1},\ \theta_Y=\gamma(\sigma_Y)\circ\gamma(\rho_Y)^{-1}$, applying $\gamma$ to (\ref{eq: phipsi}) one gets the commutative diagram in $\mathrm{Ho}(\mathcal{A})$:

$$\xymatrix@R=0.8cm{X \ar[r]^{\gamma(f)} \ar[d]_{\theta_X} & Y \ar[r]^{\gamma(g)} \ar[d]^{\theta_Y} & C(f) \ar[r]^{\gamma(h)} \ar[d]^{\gamma(\psi)\circ\gamma(\varphi)^{-1}} & \Sigma X \ar[d]^{\Sigma\theta_X} \\
RQX \ar[r]_{\gamma(\tilde{f})} & RQY \ar[r]_{\gamma(\tilde{g})} & C(\tilde{f}) \ar[r]_{\gamma(\tilde{h})} & \Sigma RQX}$$
This completes the proof.  \end{proof}

To prove (TR3), by Lemma \ref{lem:TR3} it suffices to verify the following lemma.

\begin{lem}\label{lem: TR3} \
Suppose that the both rows in the diagram below are standard triangles in $\mathrm{Ho}(\mathcal{A})$, with $X, Y, X', Y' \in \mathcal{C} \cap \mathcal{F}$, and that the left-hand square commutes in $\mathrm{Ho}(\mathcal{A})$. Then there exists a morphism $w': C(f) \longrightarrow C(f')$ such that the diagram commutes in $\mathrm{Ho}(\mathcal{A})$.
$$\xymatrix@R=0.7cm{X \ar[r]^{\gamma(f)} \ar[d]_{u'} & Y \ar[r]^{\gamma(g)} \ar[d]^{v'} & C(f) \ar[r]^{\gamma(h)} \ar@{.>}[d]^{w'} & \Sigma X \ar[d]^{\Sigma u'} \\
X' \ar[r]^{\gamma(f')} & Y' \ar[r]^{\gamma(g')} & C(f') \ar[r]^{\gamma(h')} & \Sigma X'}$$
\end{lem}

\begin{proof}  \ Since $X,\  X',\  Y,\  Y'\in \mathcal{C}\cap\mathcal{F}$, by Corollary \ref{CcapF} there are morphisms $u:X\longrightarrow X',\ v:Y\longrightarrow Y'$ in $\mathcal{A}$ such that $u'=\gamma(u),\ v'=\gamma(v)$. Hence $\gamma(f'u-vf)=\gamma(f')\gamma(u)-\gamma(v)\gamma(f)=\gamma(f')u'-v'\gamma(f)=0$. By Theorem \ref{QuillenThm} one has the equivalence $(\mathcal{C}\cap\mathcal{F})/(\mathcal{C}\cap\mathcal{F}\cap\mathcal{W})\longrightarrow\mathrm{Ho}(\mathcal{A})$, 
$[l]\mapsto\gamma(l)$. Thus $[f'u-vf]=0$, i.e., $f'u-fv: X\longrightarrow Y'$ factors through an object $W$ in $\mathcal{C}\cap\mathcal{F}\cap\mathcal{W}$:
$$\xymatrix@R=0.5cm{X\ar[rd]\ar[rr]^-{f'u-fv}&& Y' \\ & W \ar[ru]}$$
By the suspension sequence $X\overset{i}{\longrightarrow}M\overset{p}{\longrightarrow}\Sigma X\overset{\delta}{\dashrightarrow}$ and $\mathbb{E}(\Sigma X, W)=0$, 
one sees that every morphism $X\longrightarrow W$ in $\mathcal{A}$ factors through $i$. Thus, $f'u-vf:X\longrightarrow Y'$ factors through $i$, i.e., there is a morphism $\varphi: M \longrightarrow Y'$ such that $f'u-vf=\varphi\circ i$. 
Now consider the following diagram of $\mathbb{E}$-triangles in $\mathcal{A}$ (Note that the left hand parallelogram does not commute in $\mathcal{A}$, but it commutes in $\mathrm{Ho}(\mathcal{A})$):

$$\xymatrix@R=0.3cm@C=0.4cm{X \ar[rr]^{i} \ar[dr]^{u} \ar[dd]_<<{f} & & M \ar[rr]^{p} \ar[dr]^{k} \ar[dd]_<<{l} & & \Sigma X \ar@{-->}[r]^-{\delta} \ar[dr]^{\bar{u}} \ar@{=}[dd] & {} \\
& X' \ar[rr]^<<{i'} \ar[dd]_<<{f'} & & M' \ar[rr]^<<{p'} \ar[dd]_<<{l'} & & \Sigma X' \ar@{-->}[r]^-{\delta'} \ar@{=}[dd] & {} \\
Y \ar[rr]^<<{g} \ar[dr]^{v} & & C(f) \ar[rr]^<<{h}  & & \Sigma X \ar@{-->}[r]^{\varepsilon} \ar[dr]^{\bar{u}} & {} \\
& Y' \ar[rr]^{g'} & & C(f') \ar[rr]^{h'} & & \Sigma X' \ar@{-->}[r]^-{\varepsilon'} & {}}$$

\vskip5pt

\noindent where $\Sigma\gamma(u)=\gamma(\bar u)$, $f_*\delta=\varepsilon$, $f'_*\delta'=\varepsilon'$, $u_*\delta=\bar{u}^*\delta'$. Since $i_*\delta = 0$, one has
\begin{align*}
v_*\varepsilon & =v_*f_*\delta=(vf)_*\delta=(f'u-\varphi i)_*\delta \\
&=f'_*u_*\delta=f'_*\bar{u}^*\delta'=\bar{u}^*f'_*\delta' =\bar{u}^*\varepsilon'.
\end{align*}

\noindent By $(\mathrm{ET}2)$, there exists a morphism $w:C(f)\longrightarrow C(f')$ in $\mathcal{A}$ such that $wg=g'v$, \ $h'w=\bar u h$. Hence $$\gamma(w)\gamma(g)=\gamma(g')\gamma(v), \ \ \ \gamma(h')\gamma(w)=\gamma(\bar u)\gamma(h)=\Sigma\gamma(u)\gamma(h)=\Sigma u'\gamma(h).$$ Put $w'=\gamma(w): C(f)\longrightarrow C(f')$ in $\mathrm{Ho}(\mathcal{A})$. Then $$w'\gamma(g)=\gamma(g')v', \ \ \ \ \gamma(h')w'=\Sigma u'\gamma(h).$$ This completes the proof. \end{proof}

\vskip5pt

To prove \textbf{(TR4)}, we need the following two lemmas.

\begin{lem}\label{lem: TR4_1} \ Let $X\overset{f}{\longrightarrow}Y$ and $Y\overset{g}{\longrightarrow}Z$ be two morphisms in $\mathrm{Ho}(\mathcal{A})$. Then there exist cofibrations $f': X' \longrightarrow Y'$ and $g': Y' \longrightarrow Z'$ in $\mathcal{A}$, such that the following diagram
$$\xymatrix@R=0.4cm{X \ar[r]^{f} \ar[d]_{u} & Y \ar[r]^{g} \ar[d]^{v} & Z \ar[d]^{w} \\ X' \ar[r]^{\gamma(f')} & Y' \ar[r]^{\gamma(g')} & Z'}$$
commutes in $\mathrm{Ho}(\mathcal{A})$, where $u,\ v,\ w$ are isomorphisms in $\mathrm{Ho}(\mathcal{A})$.
\end{lem}

\begin{proof}  \  By Proposition \ref{prop:canonical} there is a commutative diagram
\begin{equation}\label{eq: RQXYZ}
\begin{aligned}
\xymatrix@R=0.4cm{X \ar[r]^{f} \ar[d]_{\theta_X} & Y \ar[r]^{g} \ar[d]^{\theta_Y} & Z \ar[d]^{\theta_Z} \\ RQX \ar[r]^{\gamma(\tilde{f})} & RQY \ar[r]^{\gamma(\tilde{g})} & RQZ}
\end{aligned}
\end{equation}

\noindent  in $\mathrm{Ho}(\mathcal{A})$, where $\tilde{f}$ and $\tilde{g}$ are morphisms in $\mathcal{A}$, and $\theta_X$, $\theta_Y$, $\theta_Z$ are isomorphisms in $\mathrm{Ho}(\mathcal{A})$. 
By Factorization axiom, one has $\tilde{f} = p i$ with $i \in \mathrm{CoFib}$ and $p \in \mathrm{TFib}$. By Factorization axiom, one has $\tilde{g} p = p' i'$ with $i' \in \mathrm{CoFib}$ and $p' \in \mathrm{TFib}$. Then there is a commutative diagram in $\mathcal{A}$
\begin{equation}\label{eq: RQfactoriz}
\begin{aligned}
\xymatrix@R=0.4cm{RQX \ar[r]^{\tilde{f}} \ar@{=}[d] & RQY \ar[r]^{\tilde{g}} & RQZ \\ RQX \ar[r]^{i} & U \ar[u]_{p} \ar[r]^{i'} & V. \ar[u]_{p'}}
\end{aligned}
\end{equation}

\noindent Applying  $\gamma$ to the diagram (\ref{eq: RQfactoriz}) and  (\ref{eq: RQXYZ}) one gets the commutative diagram in $\Ho(\A)$:

$$\xymatrix@R=0.5cm{X \ar[r]^{f} \ar[d]_{\theta_X} & Y \ar[r]^{g} \ar[d]^{\theta_Y} & Z \ar[d]^{\theta_Z} \\ RQX \ar[r]^{\gamma(\tilde{f})} & RQY \ar[r]^{\gamma(\tilde{g})} & RQZ \\
RQX \ar[r]^{\gamma(i)} \ar@{=}[u] & U \ar[u]_{\gamma(p)} \ar[r]^{\gamma(i')} & V \ar[u]_{\gamma(p')}}
$$

\noindent  Composing the vertical morphisms, one has the following commutative diagram in $\mathrm{Ho}(\mathcal{A})$:
$$\xymatrix@R=0.7cm{X \ar[rr]^-f \ar[d]_-{\theta_X} && Y \ar[rr]^-g \ar[d]_-{\gamma(p)^{-1}\theta_Y} && Z \ar[d]^-{\gamma(p')^{-1}\theta_Z} \\ RQX \ar[rr]_{\gamma(i)} && U \ar[rr]_{\gamma(i')} && V}$$
where the vertical morphisms are isomorphisms in $\mathrm{Ho}(\mathcal{A})$,  and $i$ and $i'$ are cofibrations. \end{proof}

Let  $X\overset{f}{\longrightarrow}Y \overset{f'}{\longrightarrow} Z'\overset{\delta_f}{\dashrightarrow}$ be an $\mathbb{E}$-triangle with $f$ a cofibration. 
Take the suspension sequence $X \overset{i}{\longrightarrow} M \overset{p}{\longrightarrow} \Sigma X\overset{\delta_X}{\dashrightarrow}$ of $X$. Applying $\mathcal{A}(-,M)$ to  $X\overset{f}{\longrightarrow}Y \longrightarrow Z'\overset{\delta_f}{\dashrightarrow}$ yields an exact sequence
\[\xymatrix{ \mathcal{A}(Y,M) \ar[r]^-{\mathcal{A}(f,M)} & \mathcal{A}(X,M) \ar[r]^-{(\delta_f)^{\#}} & \mathbb{E}(Z',M)}\]

\noindent Since $Z'\in\mathcal C$ and $M\in\mathcal F\cap\mathcal W$,  $\mathbb{E}(Z',M)=0$.  Hence there is a morphism $k:Y\longrightarrow M$ such that $kf=i$. By (ET3) there is a morphism $l_f:Z'\longrightarrow \Sigma X$ such that the diagram 
$$\xymatrix@R=0.4cm{X \ar[r]^{f} & Y \ar[r]^{f'} \ar[d]_{k} & Z' \ar@{-->}[r]^{\delta_f} \ar[d]_{l_f} & {} \\ X \ar@{=}[u] \ar[r]^{i} & M \ar[r]^{p} & \Sigma X \ar@{-->}[r]^{\delta_X} & {}}$$
of $\mathbb{E}$-triangles commutes,  and $l_f^*\delta_X=\delta_f$.

\begin{lem}\label{lem: a morphism induced by cofib} \ Let  $X\overset{f}{\longrightarrow}Y \overset{f'}{\longrightarrow} Z'\overset{\delta_f}{\dashrightarrow}$ be an $\mathbb{E}$-triangle with $f$ a cofibration. Then there is an isomorphism of triangles in $\mathrm{Ho}(\mathcal{A})$
$$\xymatrix@R=0.5cm{X \ar[r]^-{\gamma(f)} \ar@{=}[d] & Y \ar[r]^-{\gamma(g)} \ar@{=}[d] & C(f) \ar[r]^-{\gamma(h)} \ar@{.>}[d]_{\gamma(t)} & \Sigma X \ar@{=}[d] \\ X \ar[r]^{\gamma(f)} & Y \ar[r]^{\gamma(f')} & Z' \ar[r]^{-\gamma(l_f)} & \Sigma X}$$
where the first row is the standard triangle induced by $f$, and $l_f:Z'\longrightarrow \Sigma X$ satisfies $l_f^* \delta_X = \delta_f$.
\end{lem}

\begin{proof} \  By Proposition \ref{prop: pushout version 1} there is a commutative diagram of $\mathbb{E}$-triangles in $\mathcal{A}$
$$\xymatrix@R=0.4cm{X \ar[r]^{i} \ar[d]_f & M \ar[r]^{p} \ar@{.>}[d] & \Sigma X \ar@{-->}[r]^{\delta_X} \ar@{=}[d] & {} \\
Y \ar@{.>}[r]^-g \ar[d]_{f'} & C(f) \ar@{.>}[r]^{h} \ar@{.>}[d]_{t} & \Sigma X \ar@{-->}[r] & {} \\
Z' \ar@{=}[r] \ar@{-->}[d]_{\delta_f} & Z' \ar@{-->}[d] & & \\
{} & {}  & & }$$

\noindent with $t^*\delta_{f}+h^*\delta_X=0$. Since $t$ is an $\mathbb{E}$-deflation and $\mathrm{CoCone}(t)=M\in \mathcal{F}\cap \mathcal{W}$, one has $t\in \mathrm{TFib}$, and hence $\gamma(t)$ is an isomorphism in $\mathrm{Ho}(\mathcal{A})$. We claim $\gamma(-l_ft)=\gamma(h)$ in $\mathrm{Ho}(\mathcal{A})$. In fact, applying $\mathcal{A}(C(f),-)$ to $X \longrightarrow M \longrightarrow \Sigma X\overset{\delta_X}{\dashrightarrow}$ yields an exact sequence
\[\xymatrix{ \mathcal{A}(C(f),M) \ar[r]^{\mathcal{A}(C(f),p)} & \mathcal{A}(C(f),\Sigma X) \ar[r]^{(\delta_X)_{\#}} &  \mathbb{E}(C(f),X)}\]
Since $$(\delta_X)_{\#}(l_ft+h)=(l_ft+h)^*\delta_X=t^*(l_f)^*\delta_X+h^*\delta_X=t^*\delta_f+h^*\delta_X=0$$
$l_ft+h: C(f)\longrightarrow \Sigma X$ factors through $M\in \mathcal{F}\cap \mathcal{W}$. Since $M = 0$ in  $\mathrm{Ho}(\mathcal{A})$, $\gamma(-l_ft)=\gamma(h)$ in $\mathrm{Ho}(\mathcal{A})$. This justifies the claim, and completes the proof.  
\end{proof}

To verify $(\mathrm{TR}4)$, it suffices to consider the standard triangles in $\rm Ho(\mathcal A)$. By Lemmas \ref{lem: TR4_1} and \ref{lem: a morphism induced by cofib}, it suffices to consider the triangles of the form
$\xymatrix@R=0.4cm{X \ar[r]^{\gamma(f)} & Y \ar[r]^{\gamma(f')} & Z' \ar[r]^{-\gamma(l_f)} & \Sigma X}$
with $f\in \mathrm{CoFib}$. That is, it suffices to prove the following lemma.

\begin{lem}\label{lem: TR4strict} \ Let $X\overset{f}{\longrightarrow }Y \overset{i}{\longrightarrow} Z'\overset{\delta_f}{\dashrightarrow}$, \ $Y\overset{g}{\longrightarrow }Z \overset{j}{\longrightarrow} X'\overset{\delta_g}{\dashrightarrow}$, \ $X\overset{h}{\longrightarrow }Z\overset{k}{\longrightarrow} Y'\overset{\delta_h}{\dashrightarrow}$ be $\mathbb{E}$-triangles with $h=gf$, such that $f$ and $g$ are cofibrations. Suppose that $l_f:Z'\longrightarrow \Sigma X,\ l_g: X'\longrightarrow \Sigma Y,\ l_h:Y'\longrightarrow \Sigma X$ are morphisms in $\mathcal{A}$ with $$l_f^*\delta_X=\delta_f, \ \ \ l_g^*\delta_Y=\delta_g, \ \ \ l_h^*\delta_X=\delta_h.$$
Then there exist morphisms $\alpha'$ and $\beta'$ in $\mathrm{Ho}(\mathcal{A})$, such that the third column in the diagram below is in $\Delta$
$$\xymatrix@R=0.7cm{X \ar[r]^{\gamma(f)} \ar@{=}[d] & Y \ar[r]^{\gamma(i)} \ar[d]_{\gamma(g)} & Z' \ar[r]^{-\gamma(l_f)} \ar@{.>}[d]_{\alpha'} & \Sigma X \ar@{=}[d] \\
X \ar[r]^{\gamma(h)} & Z \ar[r]^{\gamma(k)} \ar[d]_{\gamma(j)} & Y' \ar[r]^{-\gamma(l_h)} \ar@{.>}[d]_{\beta'} & \Sigma X \ar[d]^{\Sigma \gamma(f)} \\
& X' \ar@{=}[r] \ar[d]_-{-\gamma(l_g)}& X' \ar[r]^{-\gamma(l_g)} \ar[d]^{(\Sigma \gamma(i))(-\gamma(l_g))} & \Sigma Y \\
& \Sigma Y \ar[r]^{\Sigma \gamma(i)} & \Sigma Z' & }$$
and the diagram commutes in $\mathrm{Ho}(\mathcal{A})$.
\end{lem}

\begin{proof} \ Since $f,\ g$ are $\mathbb{E}$-inflations, by $(\mathrm{ET}4)$ there exist morphisms  $\alpha$, $\beta$ in $\mathcal A$ such that the following diagram of $\mathbb{E}$-triangles commutes in $\mathcal A$
$$\xymatrix@R=0.6cm{X \ar[r]^{f} \ar@{=}[d] & Y \ar[r]^{i} \ar[d]_{g} & Z' \ar@{-->}[r]^{\delta_f} \ar@{.>}[d]_{\alpha} & {} \\
X \ar[r]^{h} & Z \ar[r]^{k} \ar[d]_{j} & Y' \ar@{-->}[r]^{\delta_h} \ar@{.>}[d]_{\beta} & {} \\
& X' \ar@{=}[r] \ar@{-->}[d]_{\delta_g} & X' \ar@{-->}[d]_{\delta_{\alpha}} & \\
& {}  & {} & }$$

\noindent where the third column is an $\mathbb{E}$-triangle, and $\alpha^*\delta_h=\delta_f$, $i_*\delta_g=\delta_{\alpha}$, $f_*\delta_h=\beta^*\delta_g$. Since $g\in \mathrm{CoFib}$, one has $X'\in \mathcal{C}$, and hence $\alpha\in \mathrm{CoFib}$. By Lemma \ref{lem: a morphism induced by cofib} there is a morphism $l_{\alpha}:X'\longrightarrow\Sigma Z'$ such that $l_{\alpha}^*\delta_{Z'}=\delta_{\alpha}$ and the third column in following diagram is in $\Delta$.
$$\xymatrix@R=0.8cm{X \ar[r]^{\gamma(f)} \ar@{=}[d] & Y \ar[r]^{\gamma(i)} \ar[d]_{\gamma(g)} & Z' \ar[r]^{-\gamma(l_f)} \ar@{.>}[d]_{\gamma(\alpha)} & \Sigma X \ar@{=}[d] \\
X \ar[r]^{\gamma(h)} & Z \ar[r]^{\gamma(k)} \ar[d]_{\gamma(j)} & Y' \ar[r]^{-\gamma(l_h)} \ar@{.>}[d]_{\gamma(\beta)} & \Sigma X \ar[d]^{\Sigma \gamma(f)} \\
& X' \ar@{=}[r] \ar[d]_-{-\gamma(l_g)}& X' \ar[r]^{-\gamma(l_g)} \ar[d]^{-\gamma(l_{\alpha})} & \Sigma Y \\
& \Sigma Y \ar[r]^{\Sigma \gamma(i)} & \Sigma Z' & }$$

It remains to prove that the diagram commutes in $\mathrm{Ho}(\mathcal{A})$. It is clear that $\alpha i=kg$ and $\beta k=j$.

\vskip5pt

We claim $ -\gamma(l_h)\gamma(\alpha)=-\gamma(l_f)$.
In fact, applying $\mathcal{A}(Z',-)$ to the suspension sequence $X \longrightarrow M_X  \longrightarrow \Sigma X\overset{\delta_X}{\dashrightarrow}$, one gets an exact sequence
\[\xymatrix{\mathcal{A}(Z',M_X) \ar[r] & \mathcal{A}(Z',\Sigma X) \ar[r]^{(\delta_X)_{\#}} & \mathbb{E}(Z',X)}\]
Since $$(\delta_X)_{\#}(l_h\alpha-l_f) =\left(l_h \alpha-l_f\right)^*\delta_X=\alpha^*l_h^*\delta_X-l_f^*\delta_X=\alpha^*\delta_h-\delta_f=0$$ 
it follows that $l_h \alpha-l_f$ factors through $M_X\in \mathcal{F}\cap \mathcal{W}$,  and hence $ -\gamma(l_h)\gamma(\alpha)=-\gamma(l_f)$ in  $\mathrm{Ho}(\mathcal{A})$.

\vskip5pt

Next, we claim  $-\gamma(l_{\alpha})=-(\Sigma \gamma(i))\gamma(l_g)$. In fact, applying $\mathcal{A}(X',-)$ to the suspension sequence $Z' \longrightarrow M_{Z'}  \longrightarrow \Sigma Z'\overset{\delta_{Z'}}{\dashrightarrow}$, one gets an exact sequence
\[\xymatrix{\mathcal{A}(X',M_{Z'}) \ar[r] & \mathcal{A}(X',\Sigma Z') \ar[r]^{(\delta_{Z'})_{\#}} & \mathbb{E}(X',Z')}\]
Since
\begin{align*}
(\delta_{Z'})_{\#}\left( l_{\alpha}-(\Sigma i)l_g \right)&=\left( l_{\alpha}-(\Sigma i)l_g \right)^*\delta_{Z'}=l_{\alpha}^*\delta_{Z'}-l_g^*(\Sigma i)^*\delta_{Z'}\\
&=\delta_{\alpha}-l_g^*i_*\delta_Y=\delta_{\alpha}-i_*l_g^*\delta_Y  =\delta_{\alpha}-i_*\delta_g=0
\end{align*}
then $l_{\alpha}-(\Sigma i)l_g$ factors through $M_{Z'}\in \mathcal{F}\cap \mathcal{W}$, and hence $-\gamma(l_{\alpha})=-(\Sigma \gamma(i))\gamma(l_g)$ in $\mathrm{Ho}(\mathcal{A})$ .

\vskip5pt

It remains to prove that  $-(\Sigma \gamma(f))\gamma(l_{h})=-\gamma(l_{g})\gamma(\beta)$.
Applying $\mathcal{A}(Y',-)$ to the suspension sequence $Y \longrightarrow M_Y  \longrightarrow \Sigma Y\overset{\delta_Y}{\dashrightarrow}$, one gets an exact sequence
\[\xymatrix{\mathcal{A}(Y',M_Y) \ar[r] & \mathcal{A}(Y',\Sigma Y) \ar[r]^{(\delta_Y)_{\#}} & \mathbb{E}(Y',Y)}\]
Since
\begin{align*}
(\delta_Y)_{\#}((\Sigma f)l_h-l_g \beta)&=\left( (\Sigma f)l_h-l_g \beta \right)^*\delta_Y=l_h^*(\Sigma f)^*\delta_Y-\beta^*l_g^*\delta_Y\\
& =l_h^*f_*\delta_X-\beta^*\delta_g=f_*l_h^*\delta_X-\beta^*\delta_g=f_*\delta_h-\beta^*\delta_g=0
\end{align*}
then $(\Sigma f)l_h-l_g\beta$ factors through $M_Y\in \mathcal{F}\cap \mathcal{W}$, and hence $-(\Sigma \gamma(f))\gamma(l_{h})=-\gamma(l_{g})\gamma(\beta)$ in $\mathrm{Ho}(\mathcal{A})$.
\end{proof}
This completes the proof of Theorem \ref{thm: main}.

\subsection{Relationship between the two classes of distinguished triangles} \ In this subsection, we show that the class $\Delta$ of distinguished triangles in \ref{def: disinguishedtriangle} and the class $\widetilde{\Delta}$ of distinguished triangles in Nakaoka and Palu \cite[Section 6.3]{NP} admit the relation $\widetilde{\Delta}=-\Delta$. That is, the triangle $ X\overset{f}{\longrightarrow}Y\overset{g}{\longrightarrow}Z\overset{h}{\longrightarrow}\Sigma X$ is in $\widetilde{\Delta}$ if and only if the triangle $ X\overset{-f}{\longrightarrow}Y\overset{-g}{\longrightarrow}Z\overset{-h}{\longrightarrow}\Sigma X$ is in $\Delta$, and hence there is a triangle-isomorphism $(\mathrm{Ho}(\mathcal{A}), \Sigma, \Delta) \cong (\mathrm{Ho}(\mathcal{A}), \Sigma, \widetilde{\Delta})$.

\vskip5pt

Let $(\mathrm{CoFib},\ \mathrm{Fib},\ \mathrm{Weq})$ be an admissible model structure on a weakly idempotent complete extriangulated category $(\mathcal{A},\ \mathbb{E},\ \mathfrak{s})$, and $\Sigma$ the suspension functor. Recall the definition of the class $\widetilde{\Delta}$. A triangle $X\overset{\gamma(f)}{\longrightarrow}Y\overset{\gamma(g)}{\longrightarrow}Z\overset{l(\delta)}{\longrightarrow}\Sigma X$ in $\mathrm{Ho}(\mathcal{A})$ is a {\it standard triangle}, if $X\overset{f}{\longrightarrow}Y\overset{g}{\longrightarrow}Z\overset{\delta}{\dashrightarrow}$ is an $\mathbb{E}$-triangle in $\mathcal{A}$, and $l(\delta):Z \longrightarrow  \Sigma X$ is the ``connecting morphism'' defined as follows. 
For morphisms $Z \overset{w}{\longleftarrow}D\overset{d}{\longrightarrow}\Sigma X$ in $\mathcal{A}$ satisfying 
$$w\in \mathrm{TFib}, \  \  w^*\delta=d^*\delta_X.$$
Define $l(\delta):=\gamma(d)\circ \gamma(w)^{-1}: Z\longrightarrow \Sigma X$  (see \cite[Definition 6.9]{NP}). Denote $\widetilde{\Delta}$ by the class of triangles  isomorphic to standard triangles in $\Ho(\mathcal{A})$.

\vskip5pt

Suppose that $(\mathcal{C}, [1], \mathcal{E})$ is a triangulated category. Let $-\mathcal{E}:=\{X\overset{-f}{\longrightarrow}Y\overset{-g}{\longrightarrow}Z\overset{-h}{\longrightarrow}X[1] \mid X\overset{f}{\longrightarrow}Y\overset{g}{\longrightarrow}Z\overset{h}{\longrightarrow}X[1]\in \mathcal{E}\}$. It follows from the axioms (TR1), (TR2), (TR3), and (TR4) that $(\mathcal{C}, [1], -\mathcal{E})$ is also a triangulated category, and there is a triangle isomorphism
$$(\mathrm{Id}_{\mathcal{C}}, -\mathrm{Id}_{[1]}):(\mathcal{C}, [1], \mathcal{E}) \cong (\mathcal{C}, [1], -\mathcal{E})$$
where $-\mathrm{Id}_{[1]}:[1]\longrightarrow[1]$ is the natural isomorphism: for any $X\in \mathcal{C}$, $-\mathrm{Id}_{[1]}(X) = -\mathrm{Id}_{X[1]}.$

\vskip5pt

However, it is well-known that $-\mathcal{E} \ne \mathcal{E}$ in general. In fact, there exists examples where $X\overset{f}{\longrightarrow}Y\overset{g}{\longrightarrow}Z\overset{h}{\longrightarrow}X[1]$ is in $\mathcal E$ but $X\overset{-f}{\longrightarrow}Y\overset{-g}{\longrightarrow}Z\overset{-h}{\longrightarrow} X[1]$ is \textbf{not} in $\mathcal E$ (\cite[Example 2.5.5]{ZP}).

\begin{thm}\label{thm: tri-relation}\ Let $(\mathrm{CoFib},\ \mathrm{Fib},\ \mathrm{Weq})$ be an admissible model structure on a weakly idempotent complete extriangulated category $(\mathcal{A},\ \mathbb{E},\ \mathfrak{s})$, 
and $\Sigma$ the suspension functor. Suppose that $\Delta$ is the class of triangles defined in $\ref{def: disinguishedtriangle}$, and $\widetilde{\Delta}$ is the class of triangles defined in \cite{NP}. Then $\widetilde{\Delta}=-\Delta$, and hence 
there is a triangle-isomorphism 
$$(\mathrm{Id}_{\mathrm{Ho}(\mathcal{A})}, -\mathrm{Id}_{\Sigma}):(\mathrm{Ho}(\mathcal{A}), \ \Sigma, \ \Delta) \cong (\mathrm{Ho}(\mathcal{A}), \ \Sigma, \ \widetilde{\Delta}).$$
\end{thm}

\begin{proof} \ Let $ X'\overset{f'}{\longrightarrow}Y'\overset{g'}{\longrightarrow}Z'\overset{h'}{\longrightarrow}\Sigma X'$ be a triangle in $\Delta$. Then it is isomorphic to the standard triangle $X\overset{\gamma(f)}{\longrightarrow}Y\overset{\gamma(g)}{\longrightarrow}C(f)\overset{\gamma(h)}{\longrightarrow}\Sigma X$, where $Y\overset{g}{\longrightarrow}C(f)\overset{h}{\longrightarrow}\Sigma X \overset{f_*\delta_X}{\dashrightarrow}$ is an $\mathbb{E}$-triangle in $\mathcal{A}$. Consider the following diagram of morphisms in $\mathcal{A}$:
\[\xymatrix@R=0.4cm{X \ar[r]^i \ar[d]_f & M\ar[r]^p & \Sigma X \ar@{-->}[r]^-{\delta_X} \ar@{=}[d] & \\
Y \ar[r]^g & C(f) \ar[r]^h & \Sigma X \ar@{-->}[r]^-{f_*\delta_X} & }\]

\noindent where the first row is the suspension sequence of $X$. By Proposition \ref{prop: pushout mid complete homotopy square} there is a morphism $t:M\longrightarrow C(f)$ such that the above diagram commutes and the left square is homotopy cartesian, that is,
\[\xymatrix{X\ar[r]^-{\left(\begin{smallmatrix} -f \\ i \\\end{smallmatrix}\right)} & Y\oplus M \ar[r]^-{\left(\begin{smallmatrix} g,\ t\\\end{smallmatrix}\right)} & C(f) \ar@{-->}[r]^-{h^*\delta_X} & } \]
is an $\mathbb{E}$-triangle. It follows from the definition of the ``connecting morphism'' (see Nakaoka - Palu \cite[Definition 6.9]{NP}) that $l(h^*\delta_X)=\gamma(h)$, and hence
\[\xymatrix{X\ar[r]^-{\gamma\left(\begin{smallmatrix} -f \\ i \\\end{smallmatrix}\right)} & Y\oplus M \ar[r]^-{\gamma\left(\begin{smallmatrix} g,\ t\\\end{smallmatrix}\right)} & C(f) \ar[r]^-{\gamma(h)} & \Sigma X} \]

\noindent  is a standard triangle in $\mathrm{Ho}(\mathcal{A})$ in the sense of Nakaoka - Palu \cite{NP}. Since $M\in \mathcal{F}\cap \mathcal{W}$ is a zero object in $\mathrm{Ho}(\mathcal{A})$, this standard triangle is isomorphic to
\[\xymatrix{X \ar[r]^-{-\gamma(f)} & Y \ar[r]^-{\gamma(g)} & C(f) \ar[r]^-{\gamma(h)} & \Sigma X.}\]
Since $\widetilde{\Delta}$ is closed under isomorphism, the triangle $X\overset{-\gamma(f)}{\longrightarrow}Y\overset{-\gamma(g)}{\longrightarrow}C(f)\overset{-\gamma(h)}{\longrightarrow}\Sigma X$ also belongs to $\widetilde{\Delta}$. Thus, the triangle
$$\xymatrix{X' \ar[r]^-{-f'} & Y' \ar[r]^-{-g'} & Z' \ar[r]^-{-h'} & \Sigma X}$$ belongs to $\widetilde{\Delta}$.

Conversely, let $ X'\overset{f'}{\longrightarrow}Y'\overset{g'}{\longrightarrow}Z'\overset{h'}{\longrightarrow}\Sigma X'$ be a triangle in $\widetilde{\Delta}$. Then it is isomorphic to a standard triangle $X\overset{\gamma(f)}{\longrightarrow}Y\overset{\gamma(g)}{\longrightarrow}Z\overset{l(\delta)}{\longrightarrow}\Sigma X$, where $X\overset{f}{\longrightarrow}Y\overset{g}{\longrightarrow}Z\overset{\delta}{\dashrightarrow}$ is an $\mathbb{E}$-triangle in $\mathcal{A}$ and $l(\delta):Z\longrightarrow \Sigma X$ is the connecting morphism.
By Proposition \ref{prop: pushout version 1} there is a commutative diagram of $\mathbb{E}$-triangles in $\mathcal{A}$:
$$\xymatrix@R=0.6cm{X \ar[r]^{f} \ar[d]_{i} & Y \ar[r]^{g} \ar@{.>}[d]^{s} & Z \ar@{-->}[r]^{\delta} \ar@{=}[d] & \\
M \ar@{.>}[r]^{d} \ar[d]_{p} & W \ar@{.>}[r]^{e} \ar@{.>}[d]^{t} & Z \ar@{-->}[r]^{i_*\delta} &\\
\Sigma X \ar@{=}[r] \ar@{-->}[d]_{\delta_X} & \Sigma X \ar@{-->}[d]^{f_*\delta_X}&& \\
& & & } $$

\noindent  such that $e^*\delta+t^*\delta_X=0$. Since $M\in \mathcal{F}\cap \mathcal{W}$, one has $e\in \mathrm{TFib}$, and hence $l(\delta)=-\gamma(t)\circ\gamma(e)^{-1}$. Note that \[\xymatrix{X \ar[r]^{\gamma(f)} & Y \ar[r]^{\gamma(s)} & W \ar[r]^{\gamma(t)} & \Sigma X}\]
is a standard triangle in Definition \ref{def: disinguishedtriangle}. It follows from the commutative diagram
$$\xymatrix@R=0.7cm{X \ar[r]^{\gamma(f)} \ar@{=}[d] & Y \ar[r]^{\gamma(s)} \ar@{=}[d] & W \ar[r]^{\gamma(t)} \ar[d]^{\gamma(e)} & \Sigma X \ar@{=}[d] \\ X \ar[r]^{\gamma(f)} & Y \ar[r]^{\gamma(g)} & Z \ar[r]^{-l(\delta)} & \Sigma X}$$

\noindent that $X\overset{\gamma(f)}{\longrightarrow}Y\overset{\gamma(g)}{\longrightarrow}Z\overset{-l(\delta)}{\longrightarrow}\Sigma X$ is in $\Delta$. Thus the triangle $X'\overset{f'}{\longrightarrow}Y'\overset{g'}{\longrightarrow}Z'\overset{-h'}{\longrightarrow}\Sigma X'$ is in $\Delta$.
Since $\Delta$ is closed under isomorphism, the triangle $X'\overset{-f'}{\longrightarrow}Y'\overset{-g'}{\longrightarrow}Z'\overset{-h'}{\longrightarrow}\Sigma X'$ is also in $\Delta$.
\end{proof}

\subsection{Stable categories of Frobenius extriangulated categories}  \ Similar to Frobenius categories, one can define Frobenius extriangulated categories and their stable categories. Using the approach analogous to Theorem \ref{thm: main}, one can show that the stable category of a Frobenius extriangulated category is triangulated. (Here, the extriangulated category need \textbf{not} be weakly idempotent complete.) See also Nakaoka - Palu \cite[Definition 7.1, Remark 7.5]{NP}.

\section*{\bf Disclosure Statement}

The authors state that there are no competing interests to declare.

\section*{\bf Funding}

This work was supported by the National Natural Science Foundation of China under Grant No. 12131015.

\end{document}